\documentclass[reqno,11pt]{amsart}
\usepackage{amsmath,amssymb}
\newtheorem{theorem}{Theorem}[section]
\newtheorem{corollary}[theorem]{Corollary}
\newtheorem{lemma}[theorem]{Lemma}
\newtheorem{proposition}[theorem]{Proposition}
\theoremstyle{definition}
\newtheorem{definition}[theorem]{Definition}
\theoremstyle{remark} \theoremstyle{remark}
\newtheorem{assumption}{Assumption}
\newtheorem{remark}[theorem]{Remark}
\numberwithin{equation}{section}
\DeclareMathOperator*{\esssup}{\mathrm{ess\,sup}}

\allowdisplaybreaks[3]

\title[The Microscopic Dynamics of a Spatial Ecological Model]{The Microscopic Dynamics of a Spatial Ecological Model}

\author{Yuri  Kondratiev}
\address{Fakut\"at f\"ur Mathematik, Universit\"at Bielefeld, Bielefeld D-33615, Germany and Interdisciplinary Center
for Complex Systems, Dragomanov University, Kyiv, Ukraine}
\email{kondrat@math.uni-bielefeld.de}

\author{ Yuri  Kozitsky}

\address{Instytut Matematyki, Uniwersytet Marii Curie-Sk{\l}odowskiej, 20-031 Lublin, Poland}
\email{jkozi@hektor.umcs.lublin.pl}

\begin{document}

\begin{abstract}

The evolution of states of a spatial ecological model is studied.
The model describes an infinite population of point entities placed
in $\mathbb{R}^d$ which reproduce themselves at distant points
(disperse) and die with rate that includes a competition term. The
system's states are probability measures on the space of
configurations of entities, and their evolution is described by
means of a hierarchical chain of equations for the corresponding
correlation functions derived from the Fokker-Planck equation for
measures. Under natural conditions imposed on the model parameters
it is proved that the correlation functions evolve in a scale of
Banach spaces in such a way that each correlation function
corresponds to a unique sub-Poissonian state.  Some further
properties of the evolution of states constructed in this way are
also described.

\end{abstract}
\maketitle

\tableofcontents

\section{Introduction}
\subsection{Posing}
The development of a mathematical theory of complex living systems
is a challenging task of modern applied mathematics \cite{BB}. In
this paper we continue, cf. \cite{Dima,DimaN2,FKKK,K}, studying the
model introduced in \cite{BP1,BP3,Mu}. It describes an infinite
evolving population of identical point entities (particles)
distributed over $\mathbb{R}^d$, $d\geq 1$, which reproduce
themselves and die, also due to competition. In a sense, this is the
basic individual-based model in studying large ecological
communities (e.g., of perennial plants), see \cite{O} and \cite[page
1311]{N}. As is now commonly adopted, see, e.g., \cite{BP1,BP3,O},
the appropriate mathematical context for studying models of this
kind is provided by the theory of random point fields on
$\mathbb{R}^d$ in which populations are modeled as point
configurations constituting the set
\begin{equation} \label{C1}
 \Gamma = \{\gamma\subset\mathbb R^d :
 |\gamma\cap \Lambda|<\infty\text{ for any compact $\Lambda\subset\mathbb R^d$
 }\},
\end{equation}
where $|\cdot|$ denotes cardinality. It is equipped with a
$\sigma$-field of measurable subsets that allows one  to consider
probability measures on $\Gamma$ as states of the system. To
characterize such states one employs {\it observables} --
appropriate functions $F:\Gamma \rightarrow \mathbb{R}$. Their
evolution is obtained from the Kolmogorov equation
\begin{equation}
 \label{R2}
\frac{d}{dt} F_t = L F_t , \qquad F_t|_{t=0} = F_0, \qquad t>0,
\end{equation}
where the `generator' $L$ specifies the model. The states' evolution
is then obtained from the Fokker--Planck equation
\begin{equation}
 \label{R1}
\frac{d}{dt} \mu_t = L^* \mu_t, \qquad \mu_t|_{t=0} = \mu_0,
\end{equation}
related to that in (\ref{R2}) by the duality
\begin{equation*}
\int_{\Gamma} F_0 d
  \mu_t = \int_{\Gamma} F_t d
  \mu_0.
\end{equation*}
The generator of the model studied in this paper is
\begin{eqnarray}\label{R20}
(LF)(\gamma) &=& \sum_{x\in \gamma}\left[ m + E^{-}
 (x, \gamma\setminus x) \right]\left[F(\gamma\setminus x) - F(\gamma) \right]\\[.2cm]
&& + \int_{\mathbb{R}^d} E^{+} (y, \gamma) \left[ F(\gamma\cup y) -
F(\gamma) \right]dy, \nonumber
\end{eqnarray}
where
\begin{equation}
 \label{Ra20}
E^{\pm}(x, \gamma) := \sum_{y\in \gamma} a^{\pm} (x-y).
\end{equation}
The first summand in (\ref{R20}) corresponds to the death of the
particle located at $x$ occurring independently at rate $m\geq 0$
(intrinsic mortality) and under the influence of the other particles
in $\gamma$ -- at rate $E^{-}(x, \gamma\setminus x)\geq 0$
(competition).  The second term in (\ref{R20}) describes the birth
of a particle at $y\in \mathbb{R}^d$ occurring at rate $E^{+}(y,
\gamma)\geq 0$. In the sequel, we call $a^{-}$ and $a^{+}$ {\it
competition} and {\it dispersal kernels}, respectively. This model
plays a significant role in the mathematical theory of large
ecological systems, see, e.g., \cite{O} for a detailed discussion
and the references on this matter. The version of (\ref{R20}) with
$a^{-} \equiv 0$ is the {\it continuum contact model} studied in
\cite{KKP,KS}. Having in mind the results of these works, along with
purely mathematical tasks we aim at understanding the ecological
consequences of the competition taken into account in (\ref{R20}).

The problem of constructing spatial birth and death processes in
infinite volume was first studied by R. A. Holley and D. W. Stroock
in their pioneering work \cite{HS1978}, where a special case of
nearest neighbor interactions on the real line was considered. For
more general versions of continuum birth-and-death systems, the few
results known by this time were obtained under substantial
restrictions on the birth and death rates. This relates to  the
construction of a Markov process in \cite{Gar}, as well as to the
result obtained in \cite{DimaN2} in the statistical approach (see
below). In the present work, we make an essential step forward in
studying the model specified in (\ref{R20}) assuming only that the
kernels $a^{\pm}$ satisfy some rather mild condition.

The set of finite configurations $\Gamma_0$ is a measurable subset
of $\Gamma$. If $\mu$ is such that $\mu(\Gamma_0) =1$, then the
considered system is finite in this state. If $\mu_0$ in (\ref{R1})
has such a property, the evolution $\mu_0 \mapsto \mu_t$ can be
obtained directly from (\ref{R1}), see \cite{K}. In this case
$\mu_t(\Gamma_0) =1$ for all $t>0$. States of infinite systems are
mostly such that $\mu(\Gamma_0)=0$, which makes direct solving
(\ref{R1}) with an {\it arbitrary} initial state $\mu_0$ rather
unaccessible for the method existing at this time, cf. \cite{KKM}. In this work we
continue following the statistical approach, cf.
\cite{Berns,Dima,DimaN2,FKKK,KKM}, in which the evolution of states is
described as that of the corresponding correlation functions. To
briefly explain its essence let us consider the set of all compactly
supported  continuous functions $\theta:\mathbb{R}^d\to (-1,0]$.
For a probability measure $\mu$ on
$\Gamma$ its {\it Bogoliubov} functional \cite{FKO,TobiJoao} is defined as
\begin{equation}
  \label{I1}
B_\mu (\theta) = \int_{\Gamma} \prod_{x\in \gamma} ( 1 + \theta (x))
\mu( d \gamma),
\end{equation}
with $\theta$ running through the mentioned set of functions.
For  $\pi_\varkappa$ -- the homogeneous Poisson measure with
intensity $\varkappa>0$, (\ref{I1}) takes the form
\begin{equation*}
B_{\pi_\varkappa} (\theta) = \exp\left(\varkappa
\int_{\mathbb{R}^d}\theta (x) d x \right).
\end{equation*}
In state $\pi_\varkappa$, the particles are independently
distributed over $\mathbb{R}^d$ with density $\varkappa$. The set of
{\it sub-Poissonian} states $\mathcal{P}_{\rm sP}$ is then defined
as that containing all the states $\mu$ for which $B_\mu$ can be
continued, as a function of $\theta$, to an exponential type entire
function on $L^1 (\mathbb{R}^d)$. This exactly means that $B_\mu$
can be written down in the form
\begin{eqnarray}
  \label{I3}
B_\mu(\theta) = \sum_{n=0}^\infty
\frac{1}{n!}\int_{(\mathbb{R}^d)^n} k_\mu^{(n)} (x_1 , \dots , x_n)
\theta (x_1) \cdots \theta (x_n) d x_1 \cdots d x_n,
\end{eqnarray}
where $k_\mu^{(n)}$ is the $n$-th order correlation function
corresponding to $\mu$. It is a symmetric element of $L^\infty
((\mathbb{R}^d)^n)$ for which
\begin{equation}
\label{I4}
  \|k^{(n)}_\mu \|_{L^\infty
((\mathbb{R}^d)^n)} \leq C \exp( \alpha n), \qquad n\in
\mathbb{N}_0,
\end{equation}
with some $C>0$ and $\alpha \in \mathbb{R}$. This guarantees that
$B_\mu$ is of exponential type.  One can also consider a wider class
of states, $\mathcal{P}_{\rm anal}$, by imposing the condition that
$B_\mu$ can be continued to a function on $L^1(\mathbb{R}^d)$
analytic in some neighborhood of the origin, see \cite{TobiJoao}.
In that case, the
estimate corresponding to (\ref{I4}) will contain $n!C e^{\alpha n}$
in its right-hand side. States $\mu \in \mathcal{P}_{\rm anal}$ are
characterized by strong correlations corresponding to `clustering'.
In the contact model the clustering does take place, see
\cite{KKP,KS} and especially \cite[Eq. (3.5), page 303]{Dima}.
Namely, in this model for each $t>0$ and $n\in \mathbb{N}$ the
correlation functions satisfy the following estimates
\[
{\rm const}\cdot n! c_t^n \leq k^{(n)}_t(x_1, \dots , x_n) \leq {\rm
const}\cdot n! C_t^n,
\]
where the left-hand inequality holds if all $x_i$ belong to a ball
of sufficiently small radius. If the mortality rate $m$ is big
enough, then $C_t \to 0$ as $t\to +\infty$.  That is, in the
continuum contact model the clustering persists even if the
population asymptotically dies out. With this regard, a paramount
question about the model (\ref{R20}) is whether the competition
contained in $L$ can suppress clustering. In short, the answer given
in this work is in affirmative provided the competition and
dispersal kernels satisfy a certain natural condition. They do
satisfy if $a^{-}$ is strictly positive in some vicinity of the
origin, and $a^{+}$ has finite range.

\subsection{Presenting the result}
\label{SS1.2}

In this work, for the model described in (\ref{R20}) and
(\ref{Ra20}) we obtain the following results: \vskip.1cm
\begin{itemize}
\item[ (i)] Under the condition on the kernels $a^{\pm}$ formulated in Assumption \ref{Ass1}
we prove in Theorem \ref{1tm} that the correlation functions evolve
$k_{\mu_0}^{(n)}\mapsto k_t^{(n)}$ in such a way that each
$k_t^{(n)}$ is the correlation function of a unique sub-Poissonian
measure $\mu_t$.
\item[(ii)] We give examples of the kernels $a^{\pm}$  which satisfy Assumption \ref{Ass1}.
These examples include kernels of finite range -- both short and
long dispersals (Proposition \ref{Ma2pn}),  and also Gaussian
kernels (Propositions \ref{Ma3pn}).
\item[(iii)] For the whole range of values of the intrinsic
mortality  rate $m$, in Theorem \ref{2tm} we obtain the following
bounds for the correlation functions holding for all $t\geq 0$:
\begin{eqnarray*}
&(i)& \ 0\leq k^{(n)}_t(x_1 ,\dots , x_n) \leq  C_\delta^n
\exp\left(n(\langle a^{+} \rangle -\delta)t \right), \quad 0\leq
m \leq \langle a^{+} \rangle, \\[.2cm]
&(ii)&  \ 0\leq k^{(n)}_t(x_1 ,\dots , x_n) \leq  C_\varepsilon^n
e^{-\varepsilon t}, \quad     m > \langle a^{+} \rangle,
\end{eqnarray*}
where $\langle a^{+} \rangle$ is the $L^1$-norm of $a^+$, $C_\delta$ and $C_\varepsilon$ are appropriate positive
constants, whereas $\delta< m$ and
$\varepsilon\in (\langle a^{+} \rangle, m)$ take any value in the mentioned sets. By
(\ref{I3}) these estimates give upper bounds for the type of
$B_{\mu_t}$. We describe also the pure death case where $\langle a^{+} \rangle=0$.
\end{itemize}
More detailed comments and comparison with the previous results on this model
are given in subsection \ref{SS3.3a} below.

\section{The Basic Notions}
\label{S2}

A detailed description of various aspects of the mathematical
framework of this paper can be found in
\cite{Albev,Berns,Dima,DimaN2,FKKK,Tobi,KKP,KS,Obata}. Here we
present only  some of its aspects and indicate in which of the
mentioned papers further details can be found.
 By $\mathcal{B}(\mathbb{R}^d)$ and
 $\mathcal{B}_{\rm b}(\mathbb{R}^d)$ we denote the set of all
Borel and all bounded Borel subsets of $\mathbb{R}^d$, respectively.

\subsection{The configuration spaces}
\label{SS2.1}

The space $\Gamma$ defined in (\ref{C1}) is endowed
with  the weakest topology that makes continuous all the maps
\[
\Gamma \ni \gamma \mapsto \sum_{x\in \gamma} f(x) , \quad f\in C_0
(\mathbb{R}^d).
\]
Here $C_0 (\mathbb{R}^d)$ stands for the set of all continuous
compactly supported functions $f:\mathbb{R}^d \rightarrow
\mathbb{R}$. The mentioned topology on $\Gamma$ admits a metrization
which turns it into a complete and separable metric (Polish) space.
By $\mathcal{B}(\Gamma)$ we denote the corresponding Borel
$\sigma$-field. For $n\in\mathbb{N}_0:=\mathbb{N}\cup\{0\}$, the set
of $n$-particle configurations in $\mathbb{R}^d$  is
 \[ \Gamma^{(0)} = \{ \emptyset\}, \qquad \Gamma^{(n)} = \{\eta \subset X: |\eta| = n \}, \ \ n\in \mathbb{N}.
 \]
For $n\geq 1$, $\Gamma^{(n)}$ can be identified with the
symmetrization of the set
\[
\left\{(x_1, \dots , x_n)\in \bigl(\mathbb{R}^{d}\bigr)^n: x_i \neq
x_j, \ {\rm for} \ i\neq j\right\},
\]
which allows one to introduce the topology on $\Gamma^{(n)}$ related
to the Euclidean topology of $\mathbb{R}^d$  and hence the
corresponding Borel $\sigma$-field $\mathcal{B}(\Gamma^{(n)})$. The
set of finite configurations
\begin{equation*} \Gamma_{0} :=
\bigsqcup_{n\in \mathbb{N}_0} \Gamma^{(n)}
\end{equation*}
is endowed with the topology of the disjoint union and with the
corresponding Borel $\sigma$-field $\mathcal{B}(\Gamma_{0})$. It is
a measurable subset of $\Gamma$. However, the topology just
mentioned and that induced on $\Gamma_0$ from $\Gamma$ do not
coincide.

For $\Lambda \in \mathcal{B}_{\rm b}(\mathbb{R}^d)$, the set $
\Gamma_\Lambda := \{ \gamma \in \Gamma: \gamma \subset \Lambda\}$ is
a Borel subset of  $\Gamma_0$. We equip $\Gamma_\Lambda$ with the
topology induced by that of $\Gamma_0$. Let
$\mathcal{B}(\Gamma_\Lambda)$ be the corresponding Borel
$\sigma$-field. It can be proved, see  \cite[Lemma~1.1 and
Proposition~1.3]{Obata}, that
\begin{equation*}
\mathcal{B}(\Gamma_\Lambda) = \{ \Gamma_\Lambda \cap \Upsilon :
\Upsilon \in \mathcal{B}(\Gamma)\}.
\end{equation*}
It is known \cite[page 451]{Albev} that $\mathcal{B}(\Gamma)$ is the
smallest $\sigma$-field of subsets of $\Gamma$ such that all the projections
\begin{equation}
 \label{K3}
\Gamma \ni \gamma \mapsto p_\Lambda (\gamma)= \gamma_\Lambda :=
\gamma\cap \Lambda, \qquad \Lambda \in \mathcal{B}_{\rm b}
(\mathbb{R}^d),
\end{equation}
are
$\mathcal{B}(\Gamma)/\mathcal{B}(\Gamma_\Lambda)$ measurable. This
means that $(\Gamma, \mathcal{B}(\Gamma))$ is the projective limit
of the measurable spaces $(\Gamma_\Lambda,
\mathcal{B}(\Gamma_\Lambda))$, $\Lambda \in \mathcal{B}_{\rm b}
(\mathbb{R}^d)$.
\begin{remark}
  \label{Febrk1}
From the latter discussion it follows that $\Gamma_0 \in
\mathcal{B}(\Gamma)$ and
\begin{equation}
  \label{Feb1}
 \mathcal{B}(\Gamma_0) = \{ A \cap \Gamma_0: A \in
 \mathcal{B}(\Gamma )\}.
\end{equation}
Hence, a probability measure $\mu$ on $\mathcal{B}(\Gamma)$ with the
property $\mu(\Gamma_0) = 1$ can be considered also as a measure on
$\mathcal{B}(\Gamma_0)$.
\end{remark}

\subsection{Functions  and measures on configuration spaces}
\label{SS2.2}

A Borel set $\Upsilon\subset \Gamma$ is said to be bounded if the
following holds
\[
\Upsilon \subset \bigcup_{n=0}^N \Gamma^{(n)}_\Lambda,
\]
for some $\Lambda \in \mathcal{B}_{\rm b}(\mathbb{R}^d)$ and $N\in
\mathbb{N}$. In view of (\ref{Feb1}), each bounded set is in
$\mathcal{B}(\Gamma_0)$. A function $G:\Gamma_0 \to \mathbb{R}$ is
measurable if and only if, for each $n\in \mathbb{N}$, there exists
a symmetric Borel function $G^{(n)}:(\mathbb{R}^d)^n \to \mathbb{R}$
such that
\begin{equation}
  \label{Feb2}
G(\eta) = G^{(n)} (x_1 , \dots , x_n), \quad {\rm for} \ \ \eta =
\{x_1 , \dots, x_n\}.
\end{equation}
\begin{definition}
\label{Febdf} A bounded measurable function  $G:\Gamma_0 \to
\mathbb{R}$ is said to have bounded support if:
(a) $G(\eta) = 0$ whenever $\eta \cap \Lambda^c \neq \emptyset$ for
some $\Lambda \in \mathcal{B}_{\rm b}(\mathbb{R}^d)$, $\Lambda^c
:=\mathbb{R}^d \setminus \Lambda$; (b) $G^{(n)} \equiv 0$ whenever
$n> N$ for some $N\in \mathbb{N}$. The set of all such functions is
denoted by $B_{\rm bs}(\Gamma_0)$. For a given $G\in B_{\rm bs}(\Gamma_0)$, by
$N(G)$ we denote the smallest  $N$
with the property as in (b).
\end{definition}
A map $F:\Gamma \rightarrow \mathbb{R}$ is called cylinder function
if there exist $\Lambda \in \mathcal{B}_{\rm b}(\mathbb{R}^d)$ and a
measurable $G: \Gamma_{\Lambda}\rightarrow \mathbb{R}$ such that,
cf. (\ref{K3}), $F(\gamma) = G(\gamma_\Lambda)$ for all $\gamma \in
\Gamma$. Clearly, such a map $F$ is measurable. By $\mathcal{F}_{\rm
cyl}(\Gamma)$ we denote the set of all cylinder functions. For
$\gamma \in \Gamma$, by writing $\eta \Subset \gamma$ we mean that
$\eta \subset \gamma$ and $\eta$ is finite, i.e., $\eta \in
\Gamma_0$. For $G \in B_{\rm bs}(\Gamma_0)$, we set
\begin{equation}
 \label{7A}
 (KG)(\gamma) = \sum_{\eta \Subset \gamma} G(\eta), \qquad \gamma \in \Gamma.
\end{equation}
By \cite{Tobi} we know that $K$  maps $B_{\rm bs}(\Gamma_0)$ onto
$\mathcal{F}_{\rm cyl}(\Gamma)$ and is invertible.
The Lebesgue-Poisson measure $\lambda$ on $\mathcal{B}(\Gamma_0)$ is
defined by the relation
\begin{gather}
  \label{Feb3}
\int_{\Gamma_0} G(\eta) \lambda( d \eta) = G(\emptyset) +
\sum_{n=1}^{\infty} \frac{1}{n!} \int_{(\mathbb{R}^d)^n} G^{(n)}
(x_1 , \dots, x_n) d x_1 \cdots d x_n,
\end{gather}
which has to hold for all $G\in B_{\rm bs}(\Gamma_0)$, cf.
(\ref{Feb2}). Note that $B_{\rm bs}(\Gamma_0)$ is a measure defining
class. Clearly, $\lambda(\Upsilon) < \infty$ for each bounded
$\Upsilon \in \mathcal{B}(\Gamma_0)$. With the help of (\ref{Feb3}),
we rewrite (\ref{I3}) in the following form
\begin{equation}
  \label{II3}
B_\mu (\theta) = \int_{\Gamma_0} k_\mu (\eta) \left(\prod_{x\in
\eta}\theta(x) \right)\lambda ( d \eta).
\end{equation}
In the sequel, by saying that something holds for all $\eta$ we mean
that it holds for $\lambda$-almost all $\eta\in \Gamma_0$. This
relates also to (\ref{Feb2}).

Let $\mathcal{P}(\Gamma)$, resp. $\mathcal{P}(\Gamma_0)$, stand for
the set of all probability measures on $\mathcal{B}(\Gamma)$, resp.
$\mathcal{B}(\Gamma_0)$. Note that $\mathcal{P}(\Gamma_0)$ can be
considered as a subset of $\mathcal{P}(\Gamma)$, see Remark
\ref{Febrk1}. For a given $\mu\in \mathcal{P}(\Gamma)$, the {\it
projection} $\mu^\Lambda$ is defined as
\begin{equation}
  \label{ProJ}
\mu^\Lambda (A) = \mu(p_\Lambda^{-1} (A)), \qquad A \in
\mathcal{B}(\Gamma_\Lambda),
\end{equation}
where $p_\Lambda^{-1} (A):= \{ \gamma \in \Gamma: p_\Lambda
(\gamma)\in A\}$, see (\ref{K3}). The projections of the
Lebesgue-Poisson measure $\lambda$ are defined in the same way.

Recall that  $\mathcal{P}_{\rm anal}$ (resp. $\mathcal{P}_{\rm
sP}$) denotes the set of all those $\mu\in \mathcal{P}(\Gamma)$
for each of which $B_\mu$ defined in (\ref{I1}), or (\ref{II3}),
admits continuation to a function on $L^1(\mathbb{R}^d)$ analytic in
some neighborhood of zero (resp. exponential type entire function).
The elements of $\mathcal{P}_{\rm sP}$ are called sub-Poissonian
states. One can show  \cite[Proposition 4.14]{Tobi} that for each
$\Lambda \in \mathcal{B}_{\rm b}(\mathbb{R}^d)$ and $\mu \in
\mathcal{P}_{\rm sP}$, $\mu^\Lambda$ is absolutely continuous
with respect to $\lambda^\Lambda$. The Radon-Nikodym derivative
\begin{equation}
  \label{RN}
  R_\mu^\Lambda (\eta) = \frac{d\mu^\Lambda }{d
  \lambda^\Lambda}(\eta), \qquad \eta \in \Gamma_\Lambda,
\end{equation}
and the correlation function $k_\mu$ satisfy
\begin{equation}
  \label{9AA}
  k_\mu (\eta) = \int_{\Gamma_\Lambda} R^\Lambda_\mu (\eta \cup \xi)
  \lambda^\Lambda(d\xi), \qquad \eta \in \Gamma_\Lambda,
\end{equation}
which holds for all $\Lambda \in \mathcal{B}_{\rm b}(\mathbb{R}^d)$.
Note that (\ref{9AA}) relates $R^\Lambda_\mu$ with the restriction
of $k_\mu$ to $\Gamma_\Lambda$. The fact that these are the
restrictions of one and the same function $k_\mu:\Gamma_0 \to
\mathbb{R}$ corresponds to the Kolmogorov consistency of the family
$\{\mu^\Lambda\}_{\Lambda \in \mathcal{B}_{\rm b}(\mathbb{R}^d)}$.

By (\ref{7A}), (\ref{ProJ}), and (\ref{9AA}) we get
\begin{equation*}
\int_{\Gamma} (KG) (\gamma) \mu(d \gamma) =  \langle \! \langle
G,k_\mu \rangle \! \rangle,
\end{equation*}
which holds for each $G\in B_{\rm bs}(\Gamma_0)$ and $\mu\in
\mathcal{P}_{\rm sP}$. Here and in the sequel we use the notation
\begin{equation}
  \label{dual1}
 \langle \! \langle G,k \rangle \! \rangle = \int_{\Gamma_0} G(\eta)
k(\eta)\lambda(d\eta),
\end{equation}
Define
\begin{equation}
 \label{9AY}
 B_{\rm bs}^\star (\Gamma_0) = \{ G \in B_{\rm bs}(\Gamma_0): KG \not\equiv 0,
 \   (KG)(\gamma) \geq 0 \   {\rm for}  \ {\rm all}  \ \gamma \in \Gamma\}.
\end{equation}
By  \cite[Theorems 6.1 and 6.2 and Remark 6.3]{Tobi} we know that
the following holds.
\begin{proposition}
 \label{rhopn}
Let a measurable function $k:\Gamma_0\to \mathbb{R}$ have the
following properties:
\begin{eqnarray}
 \label{9AZ}
 & &  (i)  \ \ \langle \! \langle G,k \rangle \! \rangle \geq 0 \qquad {\rm for} \ \ {\rm all} \ \ G\in
B_{\rm bs}^\star(\Gamma_0),\\[.2cm]
& & (ii) \  \ k (\emptyset)  = 1, \qquad (iii) \ \ k (\eta) \leq
C^{|\eta|},\qquad \eta\in\Gamma_0, \nonumber
\end{eqnarray}
property (iii) holding for some $C>0$. Then there exists a unique
$\mu\in \mathcal{P}_{\rm sP}$ for which $k$ is the correlation
function.
\end{proposition}
Finally, we mention the convention
\[
\sum_{a\in \emptyset} \phi_a := 0 , \qquad \prod_{a\in \emptyset}
\psi_a := 1
\]
which we use in the sequel and the integration rule, see, e.g.,
\cite{Dima},
\begin{equation}
 \label{12A}
\int_{\Gamma_0} \sum_{\xi \subset \eta} H(\xi, \eta \setminus \xi,
\eta) \lambda (d \eta) = \int_{\Gamma_0}\int_{\Gamma_0}H(\xi, \eta,
\eta\cup \xi) \lambda (d \xi) \lambda (d \eta),
\end{equation}
valid for  appropriate functions $H$.

\subsection{Spaces of functions}
\label{SS2.3} For each $\mu\in \mathcal{P}_{\rm sP}$, the
correlation function satisfies the bound (\ref{I4}) in view of which
we introduce the following Banach spaces. For $\alpha
\in\mathbb{R}$, we set
\begin{equation}
 \label{z30}
\|k\|_\alpha = \esssup_{\eta \in \Gamma_0}
|k(\eta)|\exp(-\alpha|\eta|).
\end{equation}
It is a norm, that can also be written as follows. As in
(\ref{Feb2}), each $k:\Gamma_0\to \mathbb{R}$ is defined by its
restrictions to $\Gamma^{(n)}$. Let $k^{(n)}:(\mathbb{R}^d)^n \to
\mathbb{R}$ be a symmetric Borel function such that $k^{(n)} (x_1 ,
\dots , x_n) = k(\eta)$ for $\eta = \{x_1, \dots , x_n\}$. We then
assume that $k^{(n)}\in L^\infty ((\mathbb{R}^d)^n)$, $n\in
\mathbb{N}$, cf. (\ref{I4}), and define
\begin{equation}
  \label{M}
\|k\|_\alpha = \sup_{n\in \mathbb{N}_0} e^{-\alpha n}  \nu_n(k),
\quad \nu_n(k):= \| k^{(n)}\|_{L^\infty((\mathbb{R}^d)^n)},
\end{equation}
that yields the same norm as in (\ref{z30}). Obviously,
\begin{equation}
 \label{z300}
\mathcal{K}_\alpha := \{k:\Gamma_0\to \mathbb{R}: \|k\|_\alpha <
\infty\},
\end{equation}
is a Banach space. For $\alpha' < \alpha''$, we have
$\|k\|_{\alpha''} \leq \|k\|_{\alpha'}$. Hence,
\begin{equation}
 \label{z31}
\mathcal{K}_{\alpha'} \hookrightarrow \mathcal{K}_{\alpha''}, \qquad
\ {\rm for} \ \alpha' < \alpha''.
\end{equation}
Here and in the sequel, by $X \hookrightarrow Y$ we mean a
continuous embedding of these two Banach spaces. For $\alpha\in
\mathbb{R}$, we define, cf. (\ref{9AY}) and (\ref{dual1}),
\begin{equation}
  \label{Na}
\mathcal{K}_\alpha^{\star} = \{ k \in \mathcal{K}_\alpha:
 \forall G\in B^{\star}_{\rm bs} (\Gamma_0)\
\langle \! \langle G,k \rangle \! \rangle \geq 0 \}.
\end{equation}
It is a subset of the cone
\begin{equation}
  \label{Naha}
\mathcal{K}_\alpha^{+} = \{ k \in \mathcal{K}_\alpha: k(\eta) \geq 0
\ \ {\rm for} \ \ {\rm a.a.} \ \ \eta \in \Gamma_0\}.
\end{equation}
By Proposition \ref{rhopn} we have that each $k\in
\mathcal{K}_\alpha^{\star}$ with the property $k(\emptyset) =1$ is
the correlation function of a unique $\mu \in\mathcal{P}_{\rm sP}$.
We also put
\begin{equation}
  \label{sta1}
\mathcal{K}_{\infty} = \bigcup_{\alpha \in \mathbb{R}}
\mathcal{K}_{\alpha},
\end{equation}
and equip this set with the inductive topology.
Finally, we define
\begin{equation*}
 \mathcal{K}^\star_{\infty} =
\bigcup_{\alpha \in \mathbb{R}} \mathcal{K}^\star_{\alpha}.
\end{equation*}

\section{The Model and the Results}
\label{S3}

\subsection{The model}

\label{SS3.1}

As was already mentioned, the model is specified by the expression
given in (\ref{R20}). Regarding the kernels  in (\ref{Ra20}) we
suppose that
\begin{equation}
 \label{AA}
a^{\pm} \in L^1(\mathbb{R}^d)\cap L^\infty(\mathbb{R}^d), \qquad
a^{\pm}(x) = a^{\pm}(-x) \geq 0,
\end{equation}
and thus define
\begin{equation}
 \label{14A}
\langle{a}^{\pm}\rangle = \int_{\mathbb{R}^d} a^{\pm} (x)dx,\qquad
\|a^{\pm} \|= \esssup_{x\in \mathbb{R}^d} a^{\pm}(x),
\end{equation}
and
\begin{equation}
 \label{15A}
E^{\pm} (\eta) = \sum_{x\in \eta}E^{\pm} (x,\eta\setminus x) =
\sum_{x\in \eta} \sum_{y\in \eta\setminus x}a^{\pm} (x-y), \quad
\eta \in \Gamma_0 .
\end{equation}
We also denote
\begin{equation}
 \label{16A}
E(\eta) = \sum_{x\in \eta}\left(m + E^{-} (x, \eta\setminus x)
\right) = m|\eta| + E^{-}(\eta),
\end{equation}
where $m$ is the same as in (\ref{R20}).

In addition to the standard assumptions (\ref{AA}) we shall use the
following
\begin{assumption}[$(b, \vartheta)$-{\em assumption}]
  \label{Ass1}
There exist $\vartheta>0$ and $b\geq 0$ such that the functions
introduced in (\ref{15A})  satisfy
\begin{equation}
  \label{theta}
b |\eta| + E^{-}(\eta) \geq \vartheta E^{+} (\eta), \qquad \eta \in
\Gamma_0.
\end{equation}
\end{assumption}
Note that the case of point-wise domination
\begin{equation}
 \label{dom}
a^{-} (x) \geq \vartheta a^{+} (x) , \qquad \quad x \in
\mathbb{R}^d,
\end{equation}
cf. \cite[Eq. (3.11)]{FKKK}, corresponds to (\ref{theta})  with
$b=0$. In subsection \ref{SS3.3} below we give  examples of the
kernels $a^{\pm}$ which satisfy (\ref{theta}). To exclude the
trivial case of $a^{+} = a^{-} =0$ we also assume that
\begin{equation*}
  \langle a^{-} \rangle >0.
\end{equation*}

\subsection{The results}
\label{SS.Res}

\subsubsection{The operators}
In view of the relationship between states and correlation functions
discussed in subsection \ref{SS2.3}, we describe the system's
dynamics in the following way. First we obtain the evolution
$k_{\mu_0} \mapsto k_t$  by proving the existence of a unique
solution of the Cauchy problem of the following type
\begin{equation}\label{20A}
\frac{d k_t}{d t} = L^\Delta k_t , \qquad k_t|_{t=0} = k_{\mu_0},
\end{equation}
where the action of $L^\Delta$ is calculated from (\ref{R20}).
Thereafter, we show that each $k_t$ has the property
$k_t(\emptyset)=1$ and  lies in  $\mathcal{K}_{\alpha}^{\star}$ for
some $\alpha\in \mathbb{R}$. Hence, it is the correlation function
of a unique $\mu_t\in \mathcal{P}_{\rm sP}$. This yields in turn the
evolution $\mu_0 \mapsto \mu_t$.

To describe the action of $L^\Delta$ in a systematic way we write it
in the following form, see \cite{Dima,FKKK},
\begin{equation}
{L}^\Delta = A^\Delta+B^\Delta \label{21A},
\end{equation}
where
\begin{gather}
 \label{22A}
A^\Delta = A^\Delta_1 + A^\Delta_2,\\[.2cm]
(A_1^\Delta k)(\eta) = - E(\eta) k(\eta), \qquad (A_2^\Delta
k)(\eta) = \sum_{x\in \eta} E^{+} (x,\eta\setminus x) k
(\eta\setminus x), \nonumber
\end{gather}
see also (\ref{15A}), (\ref{16A}), and
\begin{eqnarray}
B^\Delta &= & B^\Delta_1 + B^\Delta_2 , \label{23A}\\[.2cm]
(B^\Delta_1 k)(\eta) &=& - \int_{\mathbb{R}^d} E^{-} (y,\eta) k(\eta\cup y) dy ,\nonumber \\[.2cm] (B^\Delta_2 k)(\eta) &=&
 \int_{\mathbb{R}^d} \sum_{x\in \eta} a^{+} (x - y) k(\eta \setminus x \cup y) d y. \nonumber
\end{eqnarray}
The key idea of the method that we use to study (\ref{20A}) is to
employ the scale of spaces (\ref{z300}) in which $A^\Delta$ and
$B^\Delta$  act as bounded operators from $\mathcal{K}_{\alpha'}$,
to any  $\mathcal{K}_{\alpha}$ with $\alpha> \alpha'$, cf.
(\ref{z31}). For such $\alpha$ and $\alpha' $, by (\ref{z30}) and
(\ref{M}) we have, see (\ref{22A}),
\begin{eqnarray*}
 \|A^\Delta_1 k\|_\alpha &\leq & \|k\|_{\alpha'}
\esssup_{\eta \in \Gamma_0} E(\eta)\exp\left(-(\alpha- \alpha')|\eta| \right), \\[.2cm]
 \|A^\Delta_2 k\|_\alpha &\leq & \esssup_{\eta \in \Gamma_0}
e^{-\alpha |\eta|}\sum_{x\in \eta}E^{+}(x,\eta\setminus x) |k(\eta\setminus x)|\nonumber\\[.2cm]
&\leq & \|k\|_{\alpha'}e^{-\alpha'}\esssup_{\eta \in \Gamma_0}
E^{+}(\eta)\exp\left(-(\alpha- \alpha')|\eta| \right), \nonumber
\end{eqnarray*}
which by (\ref{M}) and (\ref{14A}) yields
\begin{eqnarray}
\label{ACmar}
 \|A^\Delta_1 k\|_\alpha &\leq & \|k\|_{\alpha'} \left(\frac{m}{e(\alpha - \alpha')} + \frac{4 \|a^{-} \|}{e^2 (\alpha - \alpha')^2}
 \right) \\[.2cm]
 \|A^\Delta_2 k\|_\alpha &\leq & \|k\|_{\alpha'}e^{-\alpha'} \frac{4 \|a^{+} \|}{e^2 (\alpha -
 \alpha')^2}, \nonumber
\end{eqnarray}
where we have used the estimate
\begin{equation}
  \label{Iwona}
n^p e^{-\sigma n} \leq \left( \frac{p}{e\sigma}\right)^p, \quad
p\geq 1, \ \sigma >0, \ n \in \mathbb{N}.
\end{equation}
In a similar way, we obtain from (\ref{23A}) the following estimate,
see (\ref{14A}),
\begin{equation}
  \label{Iw}
 \|B^\Delta k\|_\alpha \leq \|k\|_{\alpha'} \frac{ \langle a^{+} \rangle +  \langle a^{-} \rangle e^{\alpha'}
 }{ e(\alpha - \alpha')}.
\end{equation}
Thus, by means of (\ref{21A}) -- (\ref{23A}), and then by
(\ref{ACmar}) and (\ref{Iw}), for each $\alpha$, $\alpha' \in
\mathbb{R}$, $\alpha'<\alpha$, one can define a continuous operator
\begin{equation}
  \label{sta3}
  L^\Delta_{\alpha \alpha'} : \mathcal{K}_{\alpha'} \to
  \mathcal{K}_{\alpha}.
\end{equation}
Let $\mathcal{L}(\mathcal{K}_{\alpha'}, \mathcal{K}_{\alpha})$ stand
for the set of all bounded linear operators $\mathcal{K}_{\alpha'}
\to \mathcal{K}_{\alpha}$. The operator norm of $L^\Delta_{\alpha
\alpha'}$ can be estimated by means of the above formulas. Thus, the
family $\{L^\Delta_{\alpha \alpha'}\}_{\alpha,\alpha'}$ determines a
bounded linear operator $L^\Delta : \mathcal{K}_\infty \to
\mathcal{K}_\infty$. Along with these continuous operators, in each
$\mathcal{K}_\alpha$, $\alpha \in \mathbb{R}$, we define an
unbounded operator, $L^\Delta_\alpha$, with domain
\begin{equation}
  \label{Iw1}
\mathcal{D}^\Delta_\alpha = \{ k \in \mathcal{K}_\alpha : L^\Delta k
\in \mathcal{K}_\alpha\}\supset  \mathcal{K}_{\alpha'},
\end{equation}
which holds for each $\alpha' < \alpha$, see (\ref{ACmar}),
(\ref{Iw}), and (\ref{21A}). The operators such introduced  are
related to each other in the following way:
\begin{equation}
 \label{staS}
 \forall \alpha'< \alpha \quad
\forall k \in \mathcal{K}_{\alpha'} \qquad  L^\Delta_{\alpha
\alpha'} k = L^\Delta_{\alpha }k.
\end{equation}

\subsubsection{The statements}

Now we can make precise which equations we are going to solve. One
possibility is to consider (\ref{20A}) in a given Banach space,
$\mathcal{K}_\alpha$.
\begin{definition}
  \label{I1df}
Given $\alpha \in \mathbb{R}$ and $T\in (0, +\infty]$, by a solution
of the Cauchy problem
\begin{equation}
 \label{Cauchy}
 \frac{d}{dt} k_t = L^\Delta_\alpha k_t , \qquad k_t|_{t=0} = k_0 \in
 \mathcal{D}^\Delta_\alpha,
\end{equation}
in $\mathcal{K}_{\alpha}$ we mean a continuous map $[0,T) \ni t
\mapsto k_t \in \mathcal{D}_\alpha$, continuously differentiable in
$\mathcal{K}_{\alpha}$ on $[0,T)$ and such that (\ref{Cauchy}) is
satisfied for all $t\in [0,T)$.
\end{definition}
Another possibility is to define (\ref{20A}) in the locally convex
space (\ref{sta1}).
\begin{definition}
  \label{I2df}
For a given $T_*\in (0, +\infty]$, by a solution of the Cauchy
problem (\ref{20A}) in $\mathcal{K}_\infty$ with a given $k_0 \in
\mathcal{K}_\infty$ we mean a map $[0,T_*) \ni t \mapsto k_t \in
\mathcal{K}_\infty$, continuously differentiable on $[0,T_*)$ and
such that (\ref{20A}) is satisfied for all $t\in [0,T_*)$.
\end{definition}
Note that $T_*$ in Definition \ref{I2df} is such that for each $T<T_*$, \
there exist  $\alpha_0, \alpha
\in \mathbb{R}$, $\alpha_0 < \alpha$, for which the mentioned $k_t$
is a solution as in Definition \ref{I1df} with $k_0 \in
\mathcal{K}_{\alpha_0}$.
Our main results are contained in the following two statements.
\begin{theorem}
  \label{1tm}
Let $(b,\vartheta)$-assumption (\ref{theta}) hold true, and $\mu_0$
be an arbitrarily sub-Poissonian state. Then the problem (\ref{20A})
with $k_0 = k_{\mu_0}$ has a unique solution $k_t \in
\mathcal{K}^{\star}_{\infty}$ on the time interval $[0,+\infty)$,
which has the property $k_t(\emptyset) =1$. Therefore, for each
$t\geq 0$ there exists a unique sub-Poissonian measure $\mu_t$ such
that $k_t = k_{\mu_t}$.
\end{theorem}
The next statement describes the solutions in more detail.
\begin{theorem}
  \label{2tm}
Let $(b,\vartheta)$-assumption (\ref{theta}) hold true with $b>0$
(resp. $b=0$), and let $\alpha_0$ be such that $k_{\mu_0}\in
\mathcal{K}_{\alpha_0}$. Then the solution $k_t$  as in
Theorem \ref{1tm}, corresponding to this $k_{\mu_0}$, for all $t\geq
0$, satisfies the following estimates.
\begin{itemize}
  \item[{\it (i)}] Case $\langle a^{+} \rangle>0$ and $m\in [0,\langle a^{+}
\rangle]$: for each $\delta <m$ (resp. $\delta \leq m$) there exists
a positive $C_\delta$ such that $\log C_\delta \geq \alpha_0$ and
\begin{equation}
  \label{bed}
k_t (\eta) \leq C_\delta^{|\eta|} \exp  \left[  (\langle a^{+}
\rangle - \delta )|\eta| t \right], \qquad \eta \in \Gamma_0.
\end{equation}
  \item[{\it (ii)}] Case $\langle a^{+} \rangle>0$ and $m>\langle a^{+}
\rangle$: for each $\varepsilon \in (0,m-\langle a^{+} \rangle)$,
there exists a positive $C_\varepsilon$ such that $\log
C_\varepsilon \geq \alpha_0$ and
\begin{equation}
  \label{Bedd}
k_t (\eta) \leq C_\varepsilon^{|\eta|} \exp (-\varepsilon t) ,
\qquad \eta \neq \emptyset.
\end{equation}
\item[{\it (iii)}] Case $\langle a^{+} \rangle =0$:
\begin{equation}
  \label{Est}
  k_t (\eta) \leq  k_0(\eta) \exp \left[ - E(\eta) t \right], \qquad \eta \in \Gamma_0.
\end{equation}
\end{itemize}
If $m=0$ and $a^{-} (x) = \vartheta a^{+} (x)$, then
\begin{equation}
  \label{bed1}
k_t (\eta) = \vartheta^{-|\eta|}, \qquad t\geq 0,
\end{equation}
is a stationary solution.
\end{theorem}
The next statement relates the solution described in Theorems
\ref{1tm} and \ref{2tm} with the problem (\ref{Cauchy}), see
Definition \ref{I2df}.
\begin{corollary}
  \label{Bedcork}
In case (i) of Theorem \ref{2tm}, for each $T>0$, $k_t$ solves
(\ref{Cauchy}) in $\mathcal{K}_{\alpha_T}$ on the time interval $[0,
T)$, where
\begin{equation}
  \label{cork}
   \alpha_T = \log C_\delta + \left(\langle a^{+} \rangle -\delta
\right) T.
\end{equation}
In case (ii) (resp. (iii)),  $k_t$  solves (\ref{Cauchy}) in
$\mathcal{K}_\alpha$, $\alpha = \log C_\varepsilon$ (resp. any
$\alpha > \alpha_0$) on the time interval $[0, +\infty)$.
\end{corollary}

\subsection{Comments and comparison}

\label{SS3.3a}

\subsubsection{On the basic assumption}

By means of the  function
\begin{equation}
  \label{phi}
 \phi_\vartheta (x) = a^{-} (x) - \vartheta a^{+} (x)
\end{equation}
one can rewrite (\ref{theta}) in the following form
\[
 \sum_{x\in \eta} \sum_{y\in \eta \setminus x} \phi_\vartheta (x-y) \geq - b |\eta|, \qquad \eta \in \Gamma_0.
\]
This resembles the stability condition (with stability
constant $b\geq0$) for the interaction potential $\phi_\vartheta$
used in the statistical mechanics of continuum systems of
interacting particles, see \cite[Chapter 3]{Ruelle}. Below we employ
some techniques developed therein to prove that important classes of
the kernels $a^{\pm}$ have this property, see Propositions
\ref{Ma2pn} and \ref{Ma3pn}.

The $(b, \vartheta)$ assumption holds with $b=0$ if and only if
(\ref{dom}) does. In this case, the dispersal kernel $a^{+}$ decays
faster than the competition kernel $a^{-}$ (short dispersal). It can
be characterized as the possibility for each daughter-entity to kill
her mother-entity, or to be killed by her. In the previous works on
this model \cite{Dima,DimaN2,FKKK} the results were based on this
short dispersal condition. The novelty of the result of Proposition
\ref{Ma2pn} is that it covers also the case of long dispersal where
the range of $a^{+}$ is finite but can be bigger than that of
$a^{-}$. Noteworthy, by our Proposition \ref{Ma2pn} it follows that
the interaction potential $\Phi$ used in \cite{PechZ} is stable,
which was unknown to the authors of that paper, cf. \cite[page
146]{PechZ}. Proposition \ref{Ma3pn} describes Gaussian kernels, for
which the basic assumption is valid also for both long and short
dispersals. In this paper, we restricted our attention to the
classes of kernels described in Propositions \ref{Ma2pn} and
\ref{Ma3pn}. Extensions beyond this classes, which we plan to
realize in a separate work, can be made by means of the
corresponding methods of the statistical mechanics of interacting
particle systems.

\subsubsection{On the results}

An important feature of the results of Theorems \ref{1tm} and
\ref{2tm} is that the intrinsic mortality rate $m\geq 0$ can be
arbitrary. Theorem \ref{1tm} gives a general existence of the
evolution $\mu_0 \mapsto \mu_t$, $t>0$, in the class of
sub-Poissonian states through the evolution of the corresponding
correlation functions. Its `ecological' outcome is that the
competition in the form as in (\ref{R20}), (\ref{Ra20}) excludes
clustering provided the kernels satisfy (\ref{theta}). A complete
characterization of the evolution $k_0 \mapsto k_t$ is then given in
Theorem \ref{2tm}. By means of it this evolution is `localized' in
the spaces $\mathcal{K}_\alpha$ in Corollary \ref{Bedcork}.
According to Theorem \ref{2tm}, for $m < \langle a^{+} \rangle$, or
$m \leq \langle a^{+} \rangle$ and $b>0$ in (\ref{phi1}), the
evolution described in Theorem \ref{1tm} takes place in an
increasing scale $\{\mathcal{K}_{\alpha_T}\}_{T\geq 0}$ of the
Banach spaces introduced in (\ref{z30}) -- (\ref{z31}), cf.
(\ref{cork}). If $m
> \langle a^{+} \rangle$,  the
evolution holds in one and the same space, see  Corollary
\ref{Bedcork}. The only difference between the cases of $b>0$ and
$b=0$ is that one can take $\delta =m$ in the latter case. This
yields different results for $m = \langle a^{+} \rangle$, where the
evolution takes place in the same space $\mathcal{K}_\alpha$ with
$\alpha = \log C_m$. Note also that  for $m=0$, one should take
$\delta<0$. For $m> \langle a^{+} \rangle$, it follows from
(\ref{Bedd}) that the population dies out: for $\langle a^{+}
\rangle >0$, the following holds
\[
k^{(n)}_{\mu_t} (x_1 ,\dots , x_n) \leq e^{-\varepsilon
t}k^{(n)}_{\mu_0} (x_1 ,\dots , x_n), \quad t>0,
\]
for some $\varepsilon \in (0, m - \langle a^{+} \rangle)$, almost
all $(x_1 ,\dots , x_n)$, and each $n\in \mathbb{N}$.
 For $m>0$ and $\langle a^{+} \rangle =0$, by
(\ref{Est}) we get
\begin{equation*}
 k^{(n)}_{\mu_t} (x_1 ,\dots , x_n) \leq \exp\left(-  n m t \right) k^{(n)}_{\mu_0} (x_1 ,\dots , x_n), \quad t>0.
\end{equation*}
This means that  $ k^{(n)}_{\mu_t} (x_1 ,\dots , x_n)\to 0$ as $n\to
+\infty$ for sufficiently big $t>0$. This phenomenon does not follow
from (\ref{Bedd}). Finally, we mention that (\ref{bed1}) corresponds
to a special case of (\ref{dom}) and $m=b=0$.

\subsubsection{Comparison}

Here we compare Theorems \ref{1tm} and \ref{2tm} with the
corresponding results obtained for this model in \cite{Dima,DimaN2}
(where it was called BDLP model), and in \cite{FKKK}. Note that
these are the only works where the infinite particle version of the
model considered here was studied. In \cite{Dima,DimaN2}, the model
was supposed to satisfy the conditions, see \cite[Eqs. (3.38) and
(3.39)]{DimaN2}, which in the present notations can be formulated as
follows: (a) (\ref{dom}) holds with a given $\vartheta>0$; (b) $m
> 16 \langle a^{-} \rangle / \vartheta$ holding with the same
$\vartheta$. Under these conditions the global evolution $k_0
\mapsto k_t$  was obtained in $\mathcal{K}_\alpha$ with some
$\alpha\in \mathbb{R}$ by means of a $C_0$-semigroup. No information
was available on whether $k_t$ is a correlation function and hence
on the sign of $k_t$. In \cite{FKKK}, the restrictions were reduced
just to (\ref{dom}). Then the evolution $k_0 \mapsto k_t$ was
obtained in a scale of Banach spaces $\mathcal{K}_\alpha$ as in
Theorem \ref{1tm}, but on a bounded time interval. Also in
\cite{FKKK}, no information was obtained on whether $k_t$ is a
correlation function. Until this our work no results on the
extinction as in (\ref{Bedd}) and on the case of $a^{+} \equiv 0$
were known.

\subsection{Kernels satisfying the basic assumption}
\label{SS3.3}

Our aim now is to show that the assumption
(\ref{theta}) can be satisfied in the most of `realistic' models. We
begin, however, by establishing an important property of the kernels
satisfying (\ref{theta}). To this end we rewrite (\ref{theta}) in the
form
\begin{equation}
  \label{phi1}
\Phi_\vartheta (\eta) := \sum_{x\in \eta} \sum_{y\in \eta \setminus
x} \left[ a^{-} (x-y) - \vartheta a^{+} (x-y)\right] \geq - b
|\eta|, \qquad \eta \in \Gamma_0.
\end{equation}
\begin{proposition}
  \label{Ma1pn}
Assume that (\ref{phi1}) holds with some $\vartheta_0>0$ and
$b_0>0$. Then for each $\vartheta < \vartheta_0$, it also holds with
$b = b_0 \vartheta/ \vartheta_0$.
\end{proposition}
\begin{proof}
For $\vartheta\in (0,\vartheta_0]$, we have
\begin{equation*}
\Phi_\vartheta (\eta) = \frac{\vartheta}{\vartheta_0} \left[ \left(
\frac{\vartheta_0}{\vartheta} - 1\right)E^{-} (\eta) +
\Phi_{\vartheta_0}(\eta)\right] \geq -
\frac{\vartheta}{\vartheta_0}b_0 |\eta|,
\end{equation*}
which yields the proof.
\end{proof}
In the following two propositions we give examples of the kernels
with the property (\ref{theta}). In the first one, we assume that
the dispersal kernel has finite range, which is quite natural in
many applications. The competition kernel in turn is assumed to be
just nontrivial.
\begin{proposition}
  \label{Ma2pn}
 In addition to (\ref{AA}) and (\ref{14A}) assume that the kernels
 $a^{\pm}$ have the following properties:
\vskip.1cm
\begin{itemize}
 \item[{\it (a)}] there exist positive $c^{-}$ and $r$ such that $a^{-}
 (x) \geq c^{-}$ for $|x|< r$;
\item[{\it (b)}] there exist positive $c^{+}$ and $R$ such that $a^{+}
 (x) \leq c^{+}$ for $|x|< R$ and $a^{+}
 (x) =0$ for $|x|\geq  R$.
\end{itemize}
Then for each $b>0$, there exists $\vartheta >0$ such that
(\ref{phi1}) holds for these $b$ and $\vartheta$.
\end{proposition}
\begin{proof}
For $r\geq R$, (\ref{phi1}) holds with $b= 0$ and $\vartheta=
c^{-}/c^{+}$. Thus, it remains to consider the case $r<R$.

For $|\eta|=0$ and $|\eta|=1$, (\ref{phi1}) trivially holds with
each $b>0$ and $\vartheta>0$. For $|\eta|=2$, (\ref{phi1}) holds
whenever $\vartheta\leq b/c^{+}$. For $|\eta|>2$, we apply an
induction in $|\eta|$, similarly as it was done in \cite{Angel}. For
$x\in \eta$, we define
\[
\xi^{-}_x = \{y\in \eta: |y-x|<r\}, \quad \xi^{+}_x = \{y\in \eta: r
\leq |y-x|< R\}.
\]
Set
\begin{equation*}
 U_\vartheta (\eta) = \Phi_\vartheta (\eta) + b|\eta| = b|\eta| +  E^{-}(\eta) -
 \vartheta E^{+} (\eta).
\end{equation*}
Then the next estimate holds true for each $x\in \eta$:
\begin{eqnarray}
  \label{M2}
 U_\vartheta (x,\eta\setminus x) &:= &  U_\vartheta (\eta) -  U_\vartheta (\eta\setminus
 x)\\[.2cm] & = & b + 2 E^{-} (x, \eta\setminus x) - 2 \vartheta E^{+} (x, \eta\setminus x)
 \nonumber \\[.2cm] & \geq & b + 2(c^{-} - \vartheta c^{+})
 |\xi_x^{-}| - 2 \vartheta c^{+}
 |\xi_x^{+}|. \nonumber
\end{eqnarray}
Given $n>2$ and positive $\vartheta$ and $b$, assume that
$U_\vartheta (\eta) \geq 0$ for each $|\vartheta|=n-1$. Then to make
the inductive step by means of (\ref{M2}) we have to show that, for
each $\eta$ such that $|\eta|=n$, there exists $x\in \eta$ such that
$U_\vartheta (x,\eta\setminus x)\geq 0$.  Set
\begin{equation}
  \label{M3}
\bar{n} = |\xi^{-}_x| = \max_{y\in \eta}|\xi^{-}_y|, \qquad x\in
\eta.
\end{equation}
If $\bar{n}=0$, then $\eta$ is such that $|y-z|\geq r$ for each
distinct $y,z\in \eta$. In this case, the balls $B_z:= \{y\in
\mathbb{R}^d: |y-z|< r/2\}$,  $z\in \eta$, do not overlap. Then
$|\xi^{+}_x|\leq \Xi (d, r, R)-1 \leq  \Delta (d) (1 + 2R/r)^d-1$,
where $\Xi (d, r, R)$ is the maximum number of rigid spheres of
radius $r/2$ packed in a ball of radius $R+r/2$, and $\Delta (d)$ is
the density of the densest packing of equal rigid spheres in
$\mathbb{R}^d$, see e.g. \cite[Chapter 1]{CS}. We apply this in
(\ref{M2}) and get that $U_\vartheta (x,\eta\setminus x)\geq 0$
whenever $\vartheta \leq b/2 c^{+}(\Xi (d, r, R) - 1)$. For $\bar{n}
>0$, let $x$ be as in (\ref{M3}). Choose $y_1, \dots , y_s$ in
$\xi^+_x$ such that the balls $B_x$ and $B_{y_i}$, $i=1, \dots , s$,
realize the densest possible packing of the ball of radius $R+r/2$
centered at $x$. Then $s \leq \Xi (d, r, R)-1$ and, for each $y\in
\xi_x^{+}$, one finds $i$ such that $|y-y_i|< r$. Otherwise $B_y$
would not overlap each $B_{y_i}$, and thus the mentioned packing is
not the densest one. Therefore, the balls $C_i:=\{z\in \mathbb{R}^d:
|z-y_i|<r\}$, $i=1, \dots , s$, cover $\xi_x^{+}$. By (\ref{M3})
each $C_i$ contains $\bar{n}+1$ elements at most.  This yields
\[
|\xi_x^{+}|\leq (\bar{n}+1) (\Xi (d, r, R) -1).
\]
Now we apply this in (\ref{M2}) and obtain that $U_\vartheta
(x,\eta\setminus x)\geq 0$ for
\begin{equation*}
 \vartheta =\min\left\{\frac{c^{-}}{c^{+}\Xi (d, r, R)}; \frac{b}{2c^{+}(\Xi (d, r, R)-1)}\right\}.
\end{equation*}
Thus, the inductive step can be done, which yields the proof.
\end{proof}
\vskip.1cm  As an example of kernels with infinite range we consider
the Gaussian kernels
\begin{equation}
 \label{M5}
 a^{\pm} (x) = \frac{c_{\pm}}{(2\pi \sigma_{\pm}^2)^{d/2}}
 \exp\left(- \frac{1}{2\sigma_{\pm}^2}|x|^2 \right),
\end{equation}
where  $c_{\pm}>0$ and $\sigma_{\pm } >0$ are parameters.
\begin{proposition}
 \label{Ma3pn}
Let $a^{\pm}$ be as in (\ref{M5}). Then for each $b>0$, there exists
$\vartheta$ such that (\ref{theta}) holds for these $\vartheta$ and
$b$.
\end{proposition}
\begin{proof}
For $\sigma_{-} \geq \sigma_{+}$, we have  $a^{-}(x) \geq \vartheta
a^{+}(x)$ for all $x$ and
\[
\vartheta \leq \left(\frac{\sigma_{+} c_{-}^{1/d}}{\sigma_{-}
c_{+}^{1/d}} \right)^d.
\]
Then (\ref{phi1}), and thus (\ref{theta}), hold for such $\vartheta$
and all $b\geq 0$. For $\sigma_{-} < \sigma_{+}$, we can write, see
(\ref{phi}),
\[
 \phi_\vartheta(x) = \int_{\mathbb{R}^d} \hat{\phi}_\vartheta (k) \exp(i k\cdot x) d k,
\]
where
\begin{equation*}
\hat{\phi}_\vartheta (k) = c_{-}  \exp\left(- \frac{1}{2}
\sigma_{-}^2|k|^2 \right) \left[1 - \vartheta\frac{c_{+}}{c_{-}}
\exp\left(- \frac{1}{2}(\sigma_{+}^2 - \sigma_{-}^2)|k|^2 \right)
\right].
\end{equation*}
For $\vartheta_0= c_{-} / c_{+}$, we have that
$\hat{\phi}_{\vartheta_0} (k)\geq 0$ for all $k\in \mathbb{R}^d$.
Then $\phi_{\vartheta_0}$ is positive definite in the sense of
\cite[Section 3.2]{Ruelle}. This means that it is the Fourier
transform of a positive finite measure on $\mathbb{R}^d$, and hence
by the Bochner theorem it follows that
\[
 \sum_{x,y\in \eta} \phi_{\vartheta_0} (x-y) = \phi_{\vartheta_0} (0) |\eta| + \Phi_{\vartheta_0} (\eta) \geq 0.
\]
Thus, $\Phi_{\vartheta_0}$ satisfies (\ref{phi1}) with stability
constant $b_0= \phi_{\vartheta_0}(0)$. Then we apply Proposition
\ref{Ma1pn} and obtain that (\ref{phi1}) holds for
\[
 \vartheta = \frac{(2\pi\sigma_{-}^2)^{d/2} b}{\sigma_{+}\left(1-\left(\frac{\sigma_{-}}{\sigma_{+}}\right)^d  \right)}
\]
which completes the proof.

\end{proof}

\section{Evolution of Correlation Functions and States}

\label{S4}

In this section we prove Theorems \ref{1tm} and \ref{2tm} assuming
the validity of  Lemma \ref{Wolflm} given below. In the next section
we prove this lemma. The proof of Theorem \ref{1tm} is based on the
construction of two families of bounded operators performed in
subsection \ref{SS4.2}. By means of one of them we obtain the
solution of the problem (\ref{Cauchy}) on a bounded time interval,
similarly as it was done in \cite{FKKK}. Next, assuming that Lemma
\ref{Wolflm} holds true, and hence $k_t \geq 0$, by means of the
second family of operators we compare $k_t$ in Lemmas \ref{Biel-1lm}
with especially constructed functions and thereby prove both
Theorems \ref{1tm} and \ref{2tm}. We begin by constructing auxiliary
semigroups used to get the results of subsection \ref{SS4.2}.

\subsection{Auxiliary semigroups}

For a given $\alpha\in \mathbb{R}$, the space predual to
$\mathcal{K}_\alpha$, defined in (\ref{z300}), is
\begin{equation}
 \label{z3}
\mathcal{G}_\alpha := L^1 (\Gamma_0 , e^{\alpha |\cdot|} d \lambda),
\end{equation}
in which the norm is, cf. (\ref{Feb3}),
\begin{eqnarray}
\label{NoR}
 |G|_{\alpha} & = & \int_{\Gamma_0} |G(\eta)| \exp( \alpha
|\eta|) \lambda (d \eta) \\[.2cm] & = & \sum_{n=0}^\infty \frac{e^{\alpha n}}{n!} \|G^{(n)}\|_{L^1((\mathbb{R}^d)^n)}. \nonumber
\end{eqnarray}
Clearly, $|G|_{\alpha'} \leq |G|_{\alpha}$ for $\alpha' < \alpha$,
which yields
\begin{equation}
 \label{z5}
\mathcal{G}_{\alpha} \hookrightarrow \mathcal{G}_{\alpha'}, \qquad
{\rm for} \ \ \alpha' < \alpha,
\end{equation}
cf. (\ref{z31}). One can show that this embedding is also dense.

Recall that by $m\geq0$ we denote the mortality rate, see
(\ref{R20}). For $b\geq 0$ as in (\ref{theta}) we set
\begin{equation}
  \label{key2}
E_b(\eta) =    (b+m)|\eta| + E^{-} (\eta) = b|\eta| +E(\eta) .
\end{equation}
Here $E^{-}(\eta)$ and $E(\eta)$ are as in (\ref{15A}) and
(\ref{16A}), respectively. For the same $b$, let the action of $A_b$
on functions $G:\Gamma_0 \to \mathbb{R}$ be as follows
\begin{gather}
\label{z1}
A_b = A_{1,b} + A_2  \\[.2cm]
(A_{1,b} G) (\eta) = - E_b(\eta) G(\eta), \nonumber \\[.2cm] (A_2 G) (\eta) =
\int_{\mathbb{R}^d} E^{+} (y,\eta) G (\eta\cup y) dy. \nonumber
\end{gather}
Our aim now is to define $A_b$ as a closed unbounded operator in
$\mathcal{G}_\alpha$ the domain of which contains
$\mathcal{G}_{\alpha'}$ for any $\alpha' > \alpha$. Let
$\mathcal{G}^+_\alpha$ denote the set of all those $G\in
\mathcal{G}_\alpha$ for which $G(\eta)\geq 0$ for $\lambda$-almost
all $\eta\in \Gamma_0$. Set
\begin{equation}
  \label{IW1}
  \mathcal{D}_\alpha = \{ G \in \mathcal{G}_\alpha : E_b(\cdot) G(\cdot) \in
  \mathcal{G}_\alpha\}.
\end{equation}
For each $\alpha'
>\alpha$, $ \mathcal{D}_\alpha$ contains $\mathcal{G}_{\alpha'}$
and hence is dense in $ \mathcal{G}_\alpha$, see (\ref{z5}). Then
the first summand  in $A_b$  turns into a closed and densely defined
operator $(A_{1,b}, \mathcal{D}_\alpha)$ in $\mathcal{G}_\alpha$
such that $- A_{1,b} G \in \mathcal{G}^+_\alpha$ for each $G \in
\mathcal{D}_\alpha^{+}:= \mathcal{D}_\alpha \cap
\mathcal{G}_\alpha^+$. By (\ref{12A}) and (\ref{theta}) one gets
\begin{eqnarray}
 \label{z7}
|A_2 G |_\alpha & \leq & \int_{\Gamma_0} \int_{\mathbb{R}^d} E^{+}(y, \eta) |G(\eta\cup y)|e^{\alpha|\eta|}dy \lambda(d\eta)\\[.2cm]
& = & e^{-\alpha} \int_{\Gamma_0} |G(\eta)| e^{\alpha|\eta|}
\left(\sum_{x\in \eta} E^{+} (x, \eta \setminus x) \right) \lambda (
d \eta)
\nonumber\\[.2cm]
& = & e^{-\alpha} |E^{+} (\cdot) G(\cdot)|_\alpha  \leq
(e^{-\alpha}/\vartheta) |A_{1,b} G|_\alpha   .\nonumber
\end{eqnarray}
Then for $\alpha > - \log \vartheta$, we have that
$e^{-\alpha}/\vartheta < 1$, and hence $A_2$ is $A_{1,b}$-bounded.
This means that $(A_b,\mathcal{D}_\alpha)$ is closed and densely
defined in $\mathcal{G}_\alpha$, see (\ref{z1}).

In the proof of Lemma~\ref{1lm} below we employ the perturbation
theory for positive semigroups of operators in ordered Banach spaces
developed in \cite{TV}. Prior to stating the lemma we present the
relevant fragments of this theory in  spaces of integrable
functions. Let $E$ be a~measurable space with a~$\sigma$-finite
measure $\nu $, and $X:=L^{1}\left( E\rightarrow \mathbb{R},d\nu
\right) $ be the Banach space of $\nu $-integrable real-valued
functions on $X$ with norm $\left\Vert \cdot \right\Vert $. Let
$X^{+}$ be the cone in $X$ consisting of all $\nu $-a.e. nonnegative
functions on $E$. Clearly, $\left\Vert f+g\right\Vert =\left\Vert
f\right\Vert +\left\Vert g\right\Vert $ for any $f,g\in X^{+}$, and
 $X=X^{+}-X^{+}$. Recall that a
$C_0$-semigroup $\{S(t)\}_{t\geq0}$ of bounded linear operators on
$X$ is called \emph{positive} if $S(t)f\in X^+$ for all $f\in X^+$.
A positive semigroup is called \emph{substochastic} (resp.
\emph{stochastic}) if $\|S(t)f\| \leq \|f\|$ (resp. $\|S(t)f\| =
\|f\|$) for all $f\in X^+$. Let $\left( A_{0},D(A_{0})\right) $ be
the generator of a positive $C_{0}$ -semigroup $\{S_{0}\left(
t\right)\}_{t\geq 0}$ on $X$. Set $D^{+}(A_{0})=D(A_{0})\cap X^{+}$.
Then $D(A_{0})$ is dense in $X$, and $D^{+}(A_{0})$ is dense in
$X^{+}$. Let $P:D(A_0)\to X$ be a positive linear operator, i.e.,
$Pf\in X^+$ for all $f\in D^+(A_0)$. The next statement is an
adaptation of Theorem~2.2 in \cite{TV}.
\begin{proposition}
\label{le:substoch} Suppose that for any $f\in D^+(A_0)$, the
following holds
\begin{equation}\label{cond:substoch}
 \int_{E}\bigl( ( A_{0}+P) f\bigr) \left( x\right) \nu \left(d
x\right) \leq 0.
\end{equation}
Then for all $r\in[0,1)$, the operator $\bigl(A_0+rP, D(A_0)\bigr)$
is the generator of a substochastic $C_0$-semigroup in $X$.
\end{proposition}
\begin{lemma}\label{1lm}
For each $\alpha > - \log \vartheta$, the operator $(A_b,
\mathcal{D}_\alpha)$ is the generator of a substochastic semigroup
$\{S(t)\}_{t\geq0}$ in $\mathcal{G}_{\alpha}$.
\end{lemma}
\begin{proof}
We apply Proposition~\ref{le:substoch} with $E=\Gamma_0$,
$X=\mathcal{G}_\alpha$ as in (\ref{z3}), and $A_0=A_{1,b}$. For
$r>0$ and $A_2$ as in (\ref{z1}), we set $P=r^{-1}A_2$.  For such
$A_0$ and $P$, and for $G \in \mathcal{D}_\alpha^+$, the left-hand
side of (\ref{cond:substoch}) takes the form, cf. (\ref{z7}),
\begin{eqnarray*}
&& - \int_{\Gamma_0} E_b(\eta) G(\eta) \exp( \alpha |\eta|) \lambda (d \eta) \\[.2cm]
&& + r^{-1} \int_{\Gamma_0} \int_{\mathbb{R}^d} E^{+}(y, \eta) G(\eta\cup y) \exp(\alpha|\eta|)dy \lambda(d\eta)\nonumber \\[.2cm]
& = & \int_{\Gamma_0} \bigl( - E_b(\eta) + r^{-1} e^{-\alpha} E^{+}
(\eta)\bigr) G(\eta) \exp(\alpha|\eta|) \lambda (d\eta).\nonumber
\end{eqnarray*}
For a fixed $\alpha > - \log \vartheta$, pick $r\in (0,1)$ such that
$r^{-1} (e^{-\alpha}/\vartheta )<1$. Then, for such $\alpha$ and
$r$, we have
\begin{equation}
\label{Qq15}
 \int_{\Gamma_0} \bigl( - E_b(\eta) + r^{-1} e^{-\alpha} E^{+} (\eta)\bigr) G(\eta) \exp(\alpha|\eta|) \lambda (d\eta) \leq 0,
\end{equation}
which holds in view of (\ref{theta}). Since $r^{-1} A_2$ is a
positive operator, by Proposition~\ref{le:substoch} we have that
$A_b= A_{1,b} + A_2 = A_{1,b} + r(r^{-1} A_2)$ generates a
substochastic semigroup $\{S(t)\}_{t\geq0}$ in
$\mathcal{G}_{\alpha}$.
\end{proof}
\vskip.1cm Now we turn to constructing the
semigroup  `sun-dual' to that mentioned in Lemma \ref{1lm}. Let $A^*_b$ be the
adjoint of $(A_b, \mathcal{D}_\alpha)$ in $\mathcal{K}_\alpha$ with
domain, cf. (\ref{Iw}),
\begin{equation*}
{\rm Dom}(A^*_b) =\bigl\{k\in \mathcal{K}_\alpha: \exists
\tilde{k}\in \mathcal{K}_\alpha \ \ \forall G\in \mathcal{D}_\alpha
\ \ \langle\! \langle A_b G , k\rangle \!\rangle = \langle\! \langle
G, \tilde{k}\rangle\!\rangle\bigr\}.
\end{equation*}
For each $k\in {\rm Dom}(A^*_b)$, the action of $A^*_b$ on $k$ is
described in (\ref{22A}) with $E$ replaced by $E_b$, see
(\ref{key2}). By (\ref{ACmar}) we then get $\mathcal{K}_{\alpha'}
\subset {\rm Dom}(A^*_b)$ for each $\alpha' < \alpha$. Let
$\mathcal{Q}_\alpha$ stand for the closure of ${\rm Dom}(A^*_b)$ in
$\|\cdot\|_\alpha$. Then
\begin{equation}
 \label{z33}
\mathcal{Q}_\alpha:= \overline{{\rm Dom}(A^*_b)}\supset {\rm
Dom}(A^*_b) \supset
 \mathcal{K}_{\alpha'}, \qquad {\rm for} \ {\rm any} \ \alpha'<\alpha.
\end{equation}
Note that $\mathcal{Q}_\alpha$ is a proper subset of
$\mathcal{K}_\alpha$. For each $t\geq 0$, the adjoint $S^*(t)$ of
$S(t)$ is a bounded operator in $\mathcal{K}_\alpha$. However, the
semigroup $\{S^*(t)\}_{t\geq 0}$ is not strongly continuous. For
$t>0$, let $ S^{\odot}_\alpha(t) $ denote the restriction of
$S^*(t)$ to $\mathcal{Q}_\alpha$. Since $\{S(t)\}_{t\geq 0}$ is the
semigroup of contractions, for $k\in \mathcal{Q}_\alpha$ and all
$t\geq 0$, we have that
\begin{equation}
 \label{ACa}
\|S^{\odot}_\alpha(t)k\|_\alpha = \|S^{*}(t)k\|_\alpha \leq
\|k\|_\alpha.
\end{equation}
\begin{proposition}
 \label{1pn}
For every $\alpha' <\alpha$ and any $k\in \mathcal{K}_{\alpha'}$,
the map
\begin{equation*}
[0, +\infty) \ni t \mapsto S^{\odot}_\alpha(t) k \in
\mathcal{K}_\alpha
\end{equation*}
is continuous.
\end{proposition}
\begin{proof}
By \cite[Theorem 10.4, page 39]{Pazy}, the collection
$\{S^{\odot}_\alpha(t)\}_{t\geq 0}$ constitutes a $C_0$-semigroup on
$\mathcal{Q}_\alpha$ the generator of which, $A^{\odot}_\alpha$, is
the part of $A^*_b$ in $\mathcal{Q}_\alpha$. That is,
$A^{\odot}_\alpha$ is the restriction of $A^*_b$ to the set
\begin{equation*}
{\rm Dom} (A^{\odot}_\alpha):= \{ k\in {\rm Dom}(A^*_b): A^*_b k \in
\mathcal{Q}_\alpha\},
\end{equation*}
cf. \cite[Definition 10.3, page 39]{Pazy}. The continuity in
question follows by the $C_0$-property of the semigroup
$\{S^{\odot}_\alpha(t)\}_{t\geq 0}$ and (\ref{z33}).
\end{proof}
\vskip.1cm By (\ref{ACmar}) it follows that
\begin{equation}
 \label{AE}
{\rm Dom} (A^{\odot}_{\alpha''}) \supset \mathcal{K}_{\alpha'},
\qquad \alpha' < \alpha'',
\end{equation}
and hence, see  \cite[Theorem 2.4, page 4]{Pazy},
\begin{equation}
  \label{AEa}
S^{\odot}_{\alpha''} (t) k \in {\rm Dom} (A^{\odot}_{\alpha''}),
\end{equation}
and
\begin{equation}
  \label{AEmar}
\frac{d}{dt} S^{\odot}_{\alpha''} (t) k = A^{\odot}_{\alpha''}
S^{\odot}_{\alpha''} (t) k,
\end{equation}
which holds for all $\alpha'' \in (\alpha', \alpha]$ and $k\in
\mathcal{K}_{\alpha'}$.

\subsection{The main operators}
\label{SS4.2}

For $E_b$ as in (\ref{key2}), we set
\begin{eqnarray}
  \label{key3}
A_b^\Delta & = & A^\Delta_{1,b} + A^\Delta_2, \\[.2cm]
(A^\Delta_{1,b}k)(\eta) & = & - E_b (\eta) k(\eta), \nonumber
\end{eqnarray}
and $A^\Delta_2$ being as in (\ref{22A}). We also set
\begin{eqnarray}
  \label{key4}
B_b^\Delta & = & B^\Delta_{1} + B^\Delta_{2,b}, \\[.2cm]
(B^\Delta_{2,b}k)(\eta)  & = & (B^\Delta_2 k)(\eta) + b|\eta|k(\eta)
. \nonumber
\end{eqnarray}
Here $B^\Delta_{1}$ and $B^\Delta_{2}$ are as in (\ref{23A}). Note
that
\begin{equation}
  \label{key5}
L^\Delta= A^\Delta + B^\Delta = A^\Delta_b + B^\Delta_b.
\end{equation}
The expressions in (\ref{key3}) and (\ref{key4}) can be used to
define the corresponding continuous operators acting from
$\mathcal{K}_{\alpha'}$ to $\mathcal{K}_{\alpha}$, $\alpha'<\alpha$,
cf. (\ref{sta3}), and hence the elements of
$\mathcal{L}(\mathcal{K}_{\alpha'}, \mathcal{K}_{\alpha})$ the norms
of which are estimated by means of the analogies of (\ref{ACmar})
and (\ref{Iw}). For these operators, we use notations
$(B^\Delta_b)_{\alpha \alpha'}$ and $(B^\Delta_{2,b})_{\alpha
\alpha'}$. Then $\|(B^\Delta_b)_{\alpha \alpha'}\|$ will stand for
the operator norm, and thus (\ref{Iw}) can be rewritten in the form
\begin{equation}
\label{bed4} \|(B^\Delta_b)_{\alpha \alpha'} \| \leq \frac{\langle
a^{+} \rangle +  b + \langle a^{-} \rangle
e^{\alpha'}}{e(\alpha-\alpha')}.
\end{equation}
For fixed $\alpha > \alpha' > - \log \vartheta$,  we construct
continuous operators $Q_{\alpha\alpha'} (t;\mathbb{B}):
\mathcal{K}_{\alpha'} \to \mathcal{K}_{\alpha}$, $t>0$, which will
be used to obtain the solution $k_t$ as in Theorem \ref{1tm} and to
study its properties. Here  $\mathbb{B}$ will be taken in the
following two versions:  (a) $\mathbb{B}= B^\Delta_b$; (b)
$\mathbb{B}= B^\Delta_{2,b}$, see (\ref{key4}). In both cases, for
each $\alpha_1, \alpha_2 \in [\alpha', \alpha]$ such that $\alpha_1
< \alpha_2$, cf. (\ref{bed4}), the following holds
\begin{equation}
  \label{IwN}
 \| \mathbb{B}_{\alpha_2\alpha_1} \| \leq
 \frac{\beta(\alpha_2;\mathbb{B})}{e(\alpha_2 - \alpha_1)},
\end{equation}
with
\begin{eqnarray}
  \label{IwN1}
& & \beta(\alpha_2; B^\Delta_b)  =  \langle a^{+} \rangle +
 b + \langle a^{-} \rangle e^{\alpha_2}, \\[.2cm]
& & \quad \beta(\alpha_2; B^\Delta_{2,b})  =  \langle a^{+} \rangle
+ b. \nonumber
\end{eqnarray}
For $t>0$ and  $\alpha_1$, $\alpha_2$ as above, let
$\Sigma_{\alpha_2\alpha_1}(t):\mathcal{K}_{\alpha_1}\to
\mathcal{K}_{\alpha_2}$ be the restriction of $S^{\odot}_{\alpha_2}(t)$ to $\mathcal{K}_{\alpha_1}$, cf.
(\ref{AE}) and (\ref{AEa}).
Note that the embedding
$\mathcal{K}_{\alpha_1}\hookrightarrow \mathcal{K}_{\alpha_2}$.
can be written as
$\Sigma_{\alpha_2\alpha_1}(0)$, and hence
\begin{equation}
  \label{aKre}
\Sigma_{\alpha_2\alpha_1}(t) = \Sigma_{\alpha_2\alpha_1}(0)
S^{\odot}_{\alpha_1}(t).
\end{equation}
Also, for each $\alpha_3 > \alpha_2$, we have
\begin{equation}
  \label{norma1}
\Sigma_{\alpha_3\alpha_1}(t) = \Sigma_{\alpha_3\alpha_2}(0)
\Sigma_{\alpha_2\alpha_1}(t):= \Sigma_{\alpha_2\alpha_1}(t), \qquad
t\geq 0.
\end{equation}
Here and in the sequel, we omit writing embedding operators if no
confusing arises. In view of (\ref{ACa}), it follows that
\begin{equation}
  \label{norma}
\|\Sigma_{\alpha_2\alpha_1}(t)\| \leq 1.
\end{equation}
\begin{remark}
 \label{Nahark}
By Lemma \ref{1lm} we have that
\begin{equation*}
\forall k\in \mathcal{K}_{\alpha_{1}}^{+} \qquad
\Sigma_{\alpha_2\alpha_1}(t)k \in \mathcal{K}_{\alpha_{2}}^{+}, \ \
t\geq  0,
\end{equation*}
see (\ref{Naha}). Also $(B^\Delta_{2,b})_{\alpha_2\alpha_1}$, but
not $(B^\Delta_b)_{\alpha_2\alpha_1}$, has the same positivity
property.
\end{remark}
Set, cf. (\ref{IwN1}),
\begin{equation}
  \label{akko}
T(\alpha_2, \alpha_1;\mathbb{B}) = \frac{\alpha_2 -
\alpha_1}{\beta(\alpha_2;\mathbb{B})}, \qquad \alpha_2 > \alpha_1 ,
\end{equation}
and then
\begin{equation}
 \label{aKK}
\mathcal{A}(\mathbb{B}) = \{(\alpha_1, \alpha_2, t): - \log
\vartheta < \alpha_1 < \alpha_2, \ \ 0\leq t <   T(\alpha_2,
\alpha_1;\mathbb{B}) \}.
\end{equation}
\begin{lemma}
  \label{IN1lm}
For each of the two choices of $\mathbb{B}$, see (\ref{IwN1}), there
exists the corresponding family of linear maps, $\{Q_{\alpha_2
\alpha_1}(t;\mathbb{B}): (\alpha_1, \alpha_2, t)\in
\mathcal{A}(\mathbb{B})\}$, each element of which  has the following
properties:
\begin{itemize}
  \item[{\it (i)}] $Q_{\alpha_2
  \alpha_1}(t;\mathbb{B})\in \mathcal{L}(\mathcal{K}_{\alpha_1} , \mathcal{K}_{\alpha_2})$;
  \item[{\it (ii)}] the map $[0, T(\alpha_2, \alpha_1; \mathbb{B})) \ni t \mapsto  Q_{\alpha_2
  \alpha_1}(t;\mathbb{B}) \in \mathcal{L}(\mathcal{K}_{\alpha_1}, \mathcal{K}_{\alpha_2})
  $ is continuous;
\item[{\it (iii)}] the operator norm of $Q_{\alpha_2
  \alpha_1}(t;\mathbb{B})\in \mathcal{L}(\mathcal{K}_{\alpha_1},
\mathcal{K}_{\alpha_2})$ satisfies
\begin{equation}
  \label{Herf}
  \|Q_{\alpha_2
  \alpha_1}(t;\mathbb{B})\| \leq \frac{T(\alpha_2, \alpha_1;\mathbb{B})}{T(\alpha_2, \alpha_1;\mathbb{B})
  -t};
\end{equation}
\item[{\it (iv)}] for each $\alpha_3 \in (\alpha_1 , \alpha_2)$ and $t< T(\alpha_3, \alpha_1;\mathbb{B})$,
the following holds
\end{itemize}
\begin{equation}
  \label{bed7}
\frac{d}{dt} Q_{\alpha_2
  \alpha_1}(t;\mathbb{B}) = \bigl( \left(A^\Delta_b\right)_{\alpha_2\alpha_3} + \mathbb{B}_{\alpha_2\alpha_3}
  \bigr) Q_{\alpha_3
  \alpha_1}(t;\mathbb{B}).
\end{equation}
\end{lemma}
The proof of this lemma is based on the following construction. For
$l\in \mathbb{N}$ and $t>0$, we set
\begin{equation}
 \label{cone}
\mathcal{T}_l :=\{ (t,t_1, \dots, t_l) : 0\leq t_l \leq \cdots \leq
t_1 \leq t\},
\end{equation}
take $\alpha \in (\alpha_1, \alpha_2]$, and then take $\delta <
\alpha - \alpha_1$. Next we divide the interval $[\alpha_1,\alpha]$
into subintervals with endpoints $\alpha^s$, $s=0, \dots, 2l +1$, as
follows. Set $\alpha^0 =\alpha_1$, $\alpha^{2l+1 } = \alpha$, and
\begin{eqnarray}
\label{z40} \alpha^{2s} & = & \alpha_1 +\frac{s}{l+1}\delta + s
\epsilon, \qquad
\epsilon = (\alpha - \alpha_1 - \delta)/l, \\[.2cm]
\alpha^{2s+1} & = & \alpha_1+ \frac{s+1}{l+1}\delta + s \epsilon,
\qquad s= 0,1, \dots, l. \nonumber
\end{eqnarray}
Then for $(t,t_1, \dots, t_l) \in \mathcal{T}_l$, define
\begin{gather}
  \label{Iwo}
\Pi_{\alpha\alpha_1}^{(l)}(t,t_1, \dots , t_l;\mathbb{B}) = \Sigma_{\alpha
\alpha^{2l}}(t-t_1) \mathbb{B}_{\alpha^{2l}\alpha^{2l-1}}\times \cdots \times \\[.2cm] \times
\Sigma_{\alpha^{2s+1} \alpha^{2s}} (t_{l-s}-t_{l-s+1})
\mathbb{B}_{\alpha^{2s}\alpha^{2s-1}} \cdots \Sigma_{\alpha^{3}\alpha^{2}}
(t_{l-1}-t_{l}) \mathbb{B}_{\alpha^{2}\alpha^{1}} \Sigma_{\alpha^1 \alpha_1}
(t_l). \nonumber
\end{gather}
\begin{proposition}
  \label{N1pn}
For both choices of $\mathbb{B}$ and each $l\in \mathbb{N}$, the operators
defined in (\ref{Iwo}) have the following properties:
\begin{itemize}
\item[{\it (i)}] for each $(t, t_1 , \dots , t_l) \in \mathcal{T}_l$,
$\Pi_{\alpha\alpha_1}^{(l)}(t,t_1, \dots , t_l;\mathbb{B})\in
\mathcal{L}(\mathcal{K}_{\alpha_1}, \mathcal{K}_{\alpha})$, and the
map
\[
\mathcal{T}_l \ni (t,t_1, \dots, t_l) \mapsto
\Pi_{\alpha\alpha_1}^{(l)} (t,t_1, \dots, t_l;\mathbb{B}) \in
\mathcal{L}(\mathcal{K}_{\alpha_1}, \mathcal{K}_{\alpha})
\]
is continuous;
\item[{\it (ii)}] for fixed $t_1,t_2, \dots , t_l$, and
each $\varepsilon>0$, the map
\[
(t_1, t_1 + \varepsilon) \ni t \mapsto
\Pi_{\alpha\alpha_1}^{(l)}(t,t_1, \dots , t_l;\mathbb{B}) \in
\mathcal{L}(\mathcal{K}_{\alpha_1}, \mathcal{K}_{\alpha_2})
\]
is continuously differentiable and for each $\alpha' \in (\alpha_1,
\alpha)$ the following holds
\end{itemize}
\begin{equation}
  \label{N1p}
\frac{d}{dt} \Pi_{\alpha\alpha_1}^{(l)} (t,t_1, \dots, t_l;\mathbb{B}) =
(A^\Delta_b)_{\alpha \alpha'} \Pi_{\alpha'\alpha_1}^{(l)} (t,t_1,
\dots, t_l;\mathbb{B}).
\end{equation}
\end{proposition}
\begin{proof}
The first part of claim {\it (i)} follows by (\ref{Iwo}),
(\ref{IwN}), and (\ref{norma}). To prove the second part we apply Proposition \ref{1pn}
and (\ref{aKre}),  and then (\ref{IwN}), (\ref{IwN1}). By
(\ref{AE}), (\ref{AEmar}), and (\ref{norma1}), and the fact that
\begin{equation*}
A^{\odot}_{\alpha'}k = (A^\Delta_b)_{\alpha' \alpha} k, \qquad {\rm
for} \ \ k\in \mathcal{K}_\alpha,
\end{equation*}
one gets
\begin{equation}
  \label{bed5}
\frac{d}{dt} \Sigma_{\alpha'\alpha_{2l}} (t) =
(A^\Delta_b)_{\alpha'\alpha} \Sigma_{\alpha\alpha_{2l}} (t), \qquad
\alpha' >\alpha,
\end{equation}
which then yields (\ref{N1p}).
\end{proof}
\vskip.1cm \textit{Proof of Lemma \ref{IN1lm}.} Take any $T<
T(\alpha_2, \alpha_1;\mathbb{B})$ and then pick $\alpha \in (\alpha_1,
\alpha_2]$ and a positive $\delta< \alpha - \alpha_1$ such that
\begin{equation*}
T < T_\delta := \frac{\alpha - \alpha_1 -
\delta}{\beta(\alpha_2;\mathbb{B})}.
\end{equation*}
For this $\delta$, take $\Pi^{(l)}_{\alpha\alpha_1}$ as in
(\ref{Iwo}), and then for  set
\begin{eqnarray}
  \label{KkN}
& & Q_{\alpha\alpha_1}^{(n)} (t;\mathbb{B}) = \Sigma_{\alpha\alpha_1}(t) \\[.2cm] & & \quad +
\sum_{l=1}^n \int_0^t\int_0^{t_1} \cdots
\int_0^{t_{l-1}}\Pi_{\alpha\alpha_1}^{(l)}(t,t_1, \dots ,
t_l;\mathbb{B}) d t_l \cdots dt_1, \quad n\in \mathbb{N}. \nonumber
\end{eqnarray}
By (\ref{norma}), (\ref{IwN}), and (\ref{z40}) we have from
(\ref{Iwo}) that
\begin{equation}
  \label{Iwv}
\| \Pi_{\alpha\alpha_1}^{(l)} (t,t_1, \dots, t_l;\mathbb{B})\| \leq
\left(\frac{l}{eT_\delta}\right)^l,
\end{equation}
holding for all $l=1, \dots , n$. This yields
\begin{eqnarray}
  \label{Iv2N}
& & \| Q_{\alpha\alpha_1}^{(n)} (t;\mathbb{B}) - Q_{\alpha\alpha_1}^{(n-1)}
(t;\mathbb{B})\|
\\[.2cm] & & \quad
\leq \int_0^t\int_0^{t_1} \cdots
\int_0^{t_{n-1}}\|\Pi_{\alpha\alpha_1}^{(n)}(t,t_1, \dots ,
t_l;\mathbb{B})\| d t_n
\cdots dt_1 \nonumber \\[.2cm] & & \quad
\leq \frac{1}{n!} \left( \frac{n}{e}\right)^n \left(
\frac{T}{T_\delta}\right)^n, \nonumber
\end{eqnarray}
hence, $$\forall t\in [0,T] \quad Q_{\alpha\alpha_1}^{(n)}
(t;\mathbb{B}) \to Q_{\alpha\alpha_1}(t;\mathbb{B})\in
\mathcal{L}(\mathcal{K}_{\alpha_1},\mathcal{K}_{\alpha}), \ \ {\rm
as} \ \ n\to +\infty.$$ This proves claim {\it (i)} of the lemma.
The proof of claim {\it (ii)} follows by the fact that the mentioned
above convergence is uniform on $[0,T]$. The estimate (\ref{Herf})
readily follows from that in (\ref{Iwv}). Now by (\ref{Iwo}) and
(\ref{bed5}) we obtain
\begin{equation*}
\frac{d}{dt} Q_{\alpha_2\alpha_1}^{(n)} (t;\mathbb{B}) =
\left(A^\Delta_b\right)_{\alpha_2\alpha} Q_{\alpha\alpha_1}^{(n)}
(t;\mathbb{B}) + B_{\alpha_2\alpha} Q_{\alpha\alpha_1}^{(n-1)} (t;\mathbb{B}), \quad
n\in \mathbb{N}.
\end{equation*}
Then the continuous differentiability of the limit and (\ref{bed7})
follow by standard arguments.
\hfill%
$\square $ \vskip.1cm  \noindent
\begin{remark}
 \label{Naha1rk}
By (\ref{Iwo}), (\ref{KkN}), and Lemma \ref{IN1lm} we have that
\begin{equation}
 \label{Naha2}
\forall k\in \mathcal{K}_{\alpha_1}^+ \qquad Q_{\alpha_2
\alpha_1}(t;B_{2,b}^\Delta)k \in   \mathcal{K}_{\alpha_2}^+, \quad t
\in [0, T(\alpha_2 , \alpha_1;B_2^\Delta)).
\end{equation}
At the same time, $Q_{\alpha_2 \alpha_1}(t;B^\Delta_b)$ is not
positive, see (\ref{23A}) and Remark \ref{Nahark}.
\end{remark}

\subsection{The proof of Theorem \ref{1tm}}
\label{SS4.3}

First we prove that the problem (\ref{Cauchy}) has a unique
solution on a bounded time interval.
\begin{lemma}
  \label{B1lm}
For each $\alpha_2 > \alpha_1 > - \log \vartheta$, the problem
(\ref{Cauchy}) with $k_0 \in \mathcal{K}_{\alpha_1}$ has a unique
solution $k_t \in \mathcal{K}_{\alpha_2}$ on the time
interval $[0,T(\alpha_2, \alpha_1, B^\Delta_b))$. The solution has
the property: $k_t (\emptyset) =1$ for all $t\in [0,T(\alpha_2,
\alpha_1, B^\Delta_b))$.
\end{lemma}
\begin{proof}
For each $t\in [0,T(\alpha_2, \alpha_1, B^\Delta_b))$, one finds
$\alpha\in (\alpha_1, \alpha_2)$ such that also $t\in [0,T(\alpha,
\alpha_1, B^\Delta_b))$. Then by claim {\it (i)} of Lemma
\ref{IN1lm} and (\ref{Iw1})
\begin{equation}
  \label{fra}
k_t := Q_{\alpha \alpha_1}(t;B^\Delta_b) k_{0}
\end{equation}
lies in $\mathcal{D}^\Delta_{\alpha_2}$. By (\ref{bed7}) the
derivative of $k_t\in \mathcal{K}_{\alpha_2}$ is
\[
\frac{d}{dt} k_t = \bigr( (A^\Delta_b)_{\alpha_2\alpha} +
(B^\Delta_b)_{\alpha_2\alpha} \bigl)k_t =
L^\Delta_{\alpha_2\alpha}k_t.
\]
Hence, $k_t$ is a solution of (\ref{Cauchy}), see (\ref{staS}). Moreover, $k_t
(\emptyset) =1$ since $k_0 (\emptyset) =1$, see (\ref{9AZ}), and
\[
\left(\frac{d}{dt} k_t\right)(\emptyset) = (L^\Delta_\alpha
k_t)(\emptyset) = 0,
\]
see (\ref{21A}) -- (\ref{23A}). To prove the stated uniqueness
assume that $\tilde{k}_t \in \mathcal{D}^\Delta_{\alpha_2}$ is
another solution of (\ref{Cauchy}) with the same initial condition.
Then for each $\alpha_3>\alpha_2$, $v_t := k_t - \tilde{k}_t$ is a
solution of (\ref{Cauchy}) in $\mathcal{K}_{\alpha_3}$
with the zero initial condition. Here we assume that $t$ and $\alpha_3$
are such that  $t< T(\alpha_3, \alpha_1;B^\Delta_b)$. Clearly, $v_t$
also solves (\ref{Cauchy}) in $\mathcal{K}_{\alpha_2}$. Thus, it can be
written down in the following form
\begin{equation}
  \label{fra1}
v_t = \int_0^t \Sigma_{\alpha_3\alpha} (t-s) \left(B^\Delta_b
\right)_{\alpha \alpha_2} v_s d s,
\end{equation}
where $v_t$ on the left-hand side (resp. $v_s$ on the right-hand
side) is considered as an element of $\mathcal{K}_{\alpha_3}$ (resp.
$\mathcal{K}_{\alpha_2}$) and $\alpha \in (\alpha_2, \alpha_3)$.
Indeed, one obtains (\ref{fra1}) by integrating the equation, see
(\ref{key5}),
\[
\frac{d}{dt} v_t = L^\Delta_{\alpha_3\alpha_2} v_t =
\bigl(\left(A^\Delta_b\right)_{\alpha_3 \alpha_2}  +
\left(B^\Delta_b \right)_{\alpha_3\alpha_2} \bigr) v_t, \qquad v_0 =
0,
\]
in which the second summand is considered as a nonhomogeneous term,
see (\ref{bed5}). Let us show that for all $t<T(\alpha_2, \alpha_1;
B^\Delta_b))$, $v_t = 0$ as an element of $\mathcal{K}_{\alpha_2}$.
In view of the embedding $\mathcal{K}_{\alpha_2}\hookrightarrow
\mathcal{K}_{\alpha_3}$, cf. (\ref{z31}), this will follow from the
fact that $v_t = 0$ as an element of $\mathcal{K}_{\alpha_3}$. For a
given $n\in \mathbb{N}$, we set $\epsilon = (\alpha_3 -
\alpha_2)/2n$ and $\alpha^l = \alpha_2 + l \epsilon$, $l=0, \dots ,
2n$. Then we repeatedly apply (\ref{fra1}) and obtain
\begin{eqnarray*}
 v_t & = & \int_0^t \int_0^{t_1} \cdots \int_0^{t_{n-1}}
\Sigma_{\alpha_3\alpha^{2n-1}} (t-t_1) (B^\Delta_b)_{\alpha^{2n-1}
\alpha^{2n-2}}\times \cdots \times \\ & \times &
\Sigma_{\alpha^2\alpha^{1}} (t_{n-1}-t_n) (B^\Delta_b)_{\alpha^{1}
\alpha_2} v_{t_n} d t_n \cdots d t_1.
\end{eqnarray*}
Similarly as in (\ref{Iwv}) we then get from the latter, see
(\ref{IwN}), (\ref{IwN1}), and (\ref{norma}),
\begin{eqnarray}
 \label{fraa1}
\|v_t\|_{\alpha_3} & \leq & \frac{t^n}{n!} \prod_{l=1}^{n}
\|(B^\Delta_b)_{\alpha^{2l-1}\alpha^{2l-2}} \| \sup_{s\in [0,t]}
\|v_s\|_{\alpha_2} \\[.2cm] &\leq & \frac{1}{n!} \left( \frac{n}{e}\right)^n
\left( \frac{2 t \beta(\alpha_3; B^\Delta_b)}{\alpha_3 -
\alpha_2}\right)^n \sup_{s\in [0,t]} \|v_s\|_{\alpha_2}. \nonumber
\end{eqnarray}
This implies that $v_t =0$ for $t < (\alpha_3 - \alpha_2)/ 2
\beta(\alpha_3; B^\Delta_b)$. To prove that $v_t =0$ for all $t$ of
interest one has to repeat the above procedure appropriate number of
times.
\end{proof}
\vskip.1cm To make the next step we need the following result, the
proof of which will be done in Section \ref{S5} below.
\begin{lemma}{\it [Identification Lemma]}
  \label{Wolflm}
For each $\alpha_2 > \alpha_1 > -\log \vartheta$, there exists $\tau
(\alpha_2, \alpha_1) \in (0, T(\alpha_2, \alpha_1;B^\Delta_b))$ such
that $Q_{\alpha_2\alpha_1} (t;B^\Delta_b):
\mathcal{K}_{\alpha_1}^\star \to \mathcal{K}_{\alpha_2}^\star$ for
each $t\in [0, \tau(\alpha_2, \alpha_1)]$, see (\ref{Na}) and Lemma
\ref{IN1lm}.
\end{lemma}
In the light of Proposition \ref{rhopn},  Lemma \ref{Wolflm} claims
that for $t\in [0, \tau(\alpha_2, \alpha_1)]$, the solution $k_t$ as
in Lemma \ref{B1lm} is the correlation function of a unique
sub-Poissonian state $\mu_t$ whenever $k_0 = k_{\mu_0}$ for some
$\mu_0 \in \mathcal{P}_{\rm sP}$.

To complete the proof of Theorem \ref{1lm} we need the following
result. Recall that $\mathcal{K}^\star_\alpha \subset
\mathcal{K}^+_\alpha$, $\alpha \in \mathbb{R}$, see (\ref{Naha}).
\begin{lemma}
   \label{Biel-1lm}
Let $\alpha_2$, $\alpha_1$, and $\tau(\alpha_2, \alpha_1)$ be as in
Lemma \ref{Wolflm}. Then there exists positive $\tau_1 (\alpha_2,
\alpha_1) \leq \tau (\alpha_2, \alpha_1)$ such that, for each $t\in
[0, \tau_1(\alpha_2, \alpha_1)]$ and arbitrary $k_0\in
\mathcal{K}_{\alpha_1}^\star$ the following holds, cf. (\ref{IwN1})
and Remark \ref{Nahark},
\begin{equation}
  \label{AKre1}
0\leq \left( Q_{\alpha_2\alpha_1}(t;B^\Delta_b) k_0\right) (\eta)
\leq \left( Q_{\alpha_2\alpha_1}(t;B^\Delta_{2,b}) k_0\right)
(\eta), \qquad \eta \in \Gamma_0.
\end{equation}
\end{lemma}
\begin{proof}
The left-hand inequality in (\ref{AKre1}) follows directly by Lemma
\ref{Wolflm}. By Lemma \ref{B1lm} $k_t$ as in (\ref{fra}) solves
(\ref{Cauchy}) in $\mathcal{K}_{\alpha_2}$. Set
\[
L^\Delta_2 = A^\Delta + B^\Delta_2 = A^\Delta_b + B^\Delta_{2,b},
\]
where $A^\Delta$, $B^\Delta_2$ and $A^\Delta_b$, $B^\Delta_{2,b}$
are as in (\ref{22A}), (\ref{23A}) and (\ref{key3}), (\ref{key4}),
respectively. Then we introduce $((L^\Delta_2)_\alpha,
\mathcal{D}_\alpha^\Delta)$ and $(L^\Delta_2)_{\alpha\alpha'}$ as in
subsection \ref{SS.Res}. By claims {\it (i)} and {\it (iv)} of Lemma
\ref{IN1lm} we have that
\begin{equation}
  \label{Wo}
u_t := Q_{\alpha \alpha_1} (t;B_{2,b}^\Delta) k_{0}, \qquad  \alpha \in (\alpha_1, \alpha_2),
\end{equation}
solves the problem
\begin{equation}
  \label{aKre3}
  \frac{d}{dt} u_t = (L^\Delta_2)_{\alpha_2} u_t, \qquad u_0 = k_0,
\end{equation}
on the time interval $[0, T(\alpha_2, \alpha_1; B^\Delta_{2,b}))$.
Note that $$T(\alpha_2, \alpha_1; B^\Delta_{b}) \leq T(\alpha_2,
\alpha_1; B^\Delta_{2,b}),$$ see (\ref{IwN1}) and (\ref{akko}). Take
$\alpha, \alpha' \in (\alpha_1, \alpha_2)$, $\alpha' < \alpha$, and
pick positive $\tau_1\leq \tau(\alpha_2, \alpha_1)$ such that
\begin{equation*}
\tau_1=\tau_1(\alpha_2,\alpha_1) < \min\{T( \alpha_2, \alpha;B^\Delta_b); T(\alpha',
\alpha_1;B^\Delta_{2,b})\}.
\end{equation*}
By (\ref{aKre3}) the difference $u_t - k_t \in
\mathcal{K}_{\alpha_2}$ can be written down in the form
\begin{equation}
\label{Ju1} u_t - k_t =  \int_0^t  Q_{\alpha_2
\alpha}(t-s;B^\Delta_{2,b}) \left( - B^\Delta_1\right)_{\alpha
\alpha'} k_s ds,
\end{equation}
where $t\leq \tau_1$ and the operator $( - B^\Delta_1)_{\alpha \alpha'}$ is positive
with respect to the cone (\ref{Naha}), see (\ref{23A}) and
(\ref{Iw}).
In (\ref{Ju1}), $k_s \in \mathcal{K}_{\alpha'}$ and $Q_{\alpha_2
\alpha}(t-s;B^\Delta_{2,b}) \in
\mathcal{L}(\mathcal{K}_{\alpha},\mathcal{K}_{\alpha_2})$ for all
$s\in [0,\tau_1]$.
Since $Q_{\alpha_2 \alpha}(t-s;B^\Delta_{2,b})$ is also
positive, see Remark \ref{Nahark}, and $k_s \in
\mathcal{K}_{\alpha'}^\star \subset \mathcal{K}_{\alpha'}^+$ (by
(\ref{fra}) and Lemma \ref{Wolflm}), we have $u_t - k_t \in
\mathcal{K}_{\alpha_2}^+$ for $t\leq \tau_1(\alpha_2 , \alpha_1)$,
which yields (\ref{AKre1}).
\end{proof}
\vskip.1cm
\begin{corollary}
  \label{Sunco}
Let $\alpha_2$, $\alpha_1$, and $\tau_1(\alpha_2, \alpha_1)$ be as
in Lemma \ref{Biel-1lm}. Then  the following holds for all $t \leq
\tau_1(\alpha_2, \alpha_1)$
\begin{equation}
  \label{Sun}
\| k_t\|_{\alpha_2} = \bigl\|
Q_{\alpha_2\alpha_1}(t;B^\Delta_b)k_0\bigr\|_{\alpha_2} \leq
\frac{(\alpha_2 - \alpha_1)\|k_0\|_{\alpha_1}}{\alpha_2 - \alpha_1 -
t(\langle a^{+} \rangle + b)}.
\end{equation}
\end{corollary}
\begin{proof}
Apply (\ref{AKre1}) and then (\ref{IwN1}) and (\ref{akko}).
\end{proof}
\vskip.1cm \noindent \textit{Proof of Theorem \ref{1tm}.} Let
$\alpha_0 > - \log \vartheta$ be such that $k_{\mu_0} \in
\mathcal{K}_{\alpha_0}$, cf, (\ref{z31}). Then by  Lemma \ref{B1lm}
we have that for each $\alpha_1 > \alpha_0$ and $\alpha\in (\alpha_0, \alpha_1)$,
\[
k_t := Q_{\alpha\alpha_0}(t;B^\Delta_b)k_0 \in
\mathcal{K}_{\alpha}^\star, \qquad t\leq \tau_1 (\alpha_1,
\alpha_0),
\]
solves (\ref{Cauchy}) in $\mathcal{K}_{\alpha_1}$. Its continuation
to an arbitrary $t>0$ follows by (\ref{Sun}) in a standard way.
\hfill%
$\square $

\subsection{The proof of Theorem  \ref{2tm}}

\label{SS4.4}

\subsubsection{Case $\langle a^{+} \rangle >0$ and $m\in [0,\langle a^{+} \rangle
]$}

The proof will be done by picking the corresponding bounds for $u_t$
defined in (\ref{Wo}) with $k_0 = k_{\mu_0} \in
\mathcal{K}_{\alpha_0}^\star$. Recall that, for $\alpha_1 >
\alpha_0$, $u_t\in \mathcal{K}_{\alpha_1}$ for $t< T(\alpha_1,
\alpha_0;B^\Delta_{2,b})$. For a given $\delta \leq m$, let us
choose the value of $C_\delta$. The first condition is that
\begin{equation}
  \label{Sun4}
  C_\delta^{|\eta|} \geq k_0 (\eta).
\end{equation}
Next, if (\ref{theta}) holds with a given $\vartheta>0$ and $b=0$,
we take any $\delta \leq m$ and $C_\delta \geq 1/\vartheta$ such
that also (\ref{Sun4}) holds. If (\ref{theta}) holds with $b>0$, we
take any $\delta < m$ and then  $C_\delta \geq
b/(m-\delta)\vartheta$ such that also (\ref{Sun4}) holds. In all
this cases, by Proposition \ref{Ma1pn} we have that
\begin{equation}
  \label{Sun3}
E^{-} (\eta) - \frac{1}{C_\delta} E^{+} (\eta) \geq - (m-\delta)
|\eta|, \qquad \eta \in \Gamma_0.
\end{equation}
Let $r_t(\eta)$ denote the right-hand side of (\ref{bed}). For
$\alpha_1 > \alpha_0$, we take $\alpha, \alpha' \in (\alpha_0,
\alpha_1)$, $\alpha' < \alpha$ and then consider
\begin{eqnarray}
  \label{Ju4}
v_t &:= & Q_{\alpha_1 \alpha_0}(t;B^\Delta_{2,b})r_0 \\[.2cm]
&=& r_t + \int_0^t Q_{\alpha_1
\alpha}(t-s;B^\Delta_{2,b})D_{\alpha\alpha'} r_s ds, \nonumber
\end{eqnarray}
where
\begin{equation}
  \label{Ju3a}
t\leq \tau_2 := \min\left\{\frac{\alpha'- \alpha_0}{\langle a^{+} \rangle -
\delta}  ;T(\alpha_1, \alpha;B^\Delta_{2,b})\right\}.
\end{equation}
The operator $D$ in (\ref{Ju4}) is
\begin{eqnarray}
  \label{Ju5}
(D_{\alpha\alpha'} r_s) (\eta) & = & \bigg{[} - m|\eta| -
E^{-}(\eta) + \frac{1}{C_\delta} \exp\bigg{(}- (\langle a^{+}
\rangle -
\delta)s\bigg{)} E^{+}(\eta) \nonumber \\[.2cm] & + &
\delta |\eta| \bigg{]} r_s (\eta)  \leq  0, \qquad \eta \in
\Gamma_0.
\end{eqnarray}
The latter inequality holds  for all $s\in [0,\tau_2]$,  see (\ref{Sun3}), and all
$m\in [0, \langle a^{+} \rangle]$ and  $\delta < m$. Then by
(\ref{Naha2}) we obtain from (\ref{Wo}), the first line of
(\ref{Ju4}), and (\ref{Sun4}) that
\[
u_t ( \eta) \leq v_t (\eta), \qquad t< T(\alpha_1,
\alpha_0;B^\Delta_{2,b}).
\]
Then by the second line of (\ref{Ju4}) and (\ref{Ju5}) we get that
for $t\leq \tau_2$, see (\ref{Ju3a}), the following holds
\[
u_t (\eta) \leq v_t(\eta) \leq r_t(\eta) , \qquad \eta \in \Gamma_0.
\]
The continuation of the latter inequality to bigger values of $t$ is
straightforward. This completes the proof for this case.

\subsubsection{Case $\langle a^{+} \rangle >0$ and $m >\langle a^{+} \rangle
$}

Take $\varepsilon \in (0, m - \langle a^{+} \rangle)$ and then set
\begin{equation*}
\vartheta_\varepsilon = \vartheta \left(1 - \frac{\varepsilon + 2
\langle a^{+} \rangle}{2 m} \right).
\end{equation*}
Thereafter, choose $C_\varepsilon\geq 1/\vartheta_\varepsilon$ such
that
\[
C_\varepsilon^{|\eta|} \geq k_0(\eta), \qquad \eta \in \Gamma_0.
\]
Then, cf. (\ref{Sun3}),
\begin{equation}
 \label{Ju6}
E^{-} (\eta) - \frac{1}{C_\varepsilon}  E^{+} (\eta) \geq
-(m-\langle a^{+} \rangle - \varepsilon/2) |\eta|, \qquad \eta \in
\Gamma_0.
\end{equation}
Let now $r_t$ stand for the right-hand side of (\ref{Bedd}). Then
the second line of (\ref{Ju4}) holds with $D_{\alpha\alpha'}$
replaced by $D_{\alpha\alpha'}^\varepsilon$. By definition the
latter is such that: (a) $(D_{\alpha\alpha'}^\varepsilon
r_s)(\emptyset)=0$;
\[
 {\rm (b)} \quad (D^\varepsilon_{\alpha\alpha'}r_s)(\{x\}) = - (m - \langle a^{+} \rangle - \varepsilon) r_s (\{x\}) \leq 0,
\]
and, for $|\eta| \geq 2$, see (\ref{Ju6}),
\begin{eqnarray*}
 {\rm (c)} \quad(D^\varepsilon_{\alpha\alpha'}r_s)(\eta) & = & \left[ \varepsilon - m|\eta| - E^{-} (\eta) + \frac{1}{C_\varepsilon}E^{+} (\eta) +
\langle a^{+} \rangle |\eta| \right]r_s (\eta)\\[.2cm]
& \leq & \varepsilon\left( 1 - |\eta|/2\right) r_s(\eta ) \leq 0.
\end{eqnarray*}
This yields (\ref{Bedd}) and thus completes the proof for this case.

\subsubsection{The remaining cases}
For $\langle a^{+} \rangle =0$ and $t>0$, we set
\begin{equation}
  \label{LL}
\left(Q^{(0)}_{\alpha \alpha'} (t)u \right) (\eta) = \exp\left[ - t
E(\eta) \right] u(\eta),
\end{equation}
where $\alpha' < \alpha$ and $u\in \mathcal{K}_{\alpha'}$. Then, cf.
Lemma \ref{IN1lm}, $Q^{(0)}_{\alpha \alpha'} (t):
\mathcal{K}_{\alpha'} \to \mathcal{K}_{\alpha}$ continuously, and the map
$$[0, +\infty) \ni t \mapsto Q^{(0)}_{\alpha
\alpha'}(t) \in \mathcal{L}( \mathcal{K}_{\alpha'} ,
\mathcal{K}_{\alpha})$$ is continuous and such that, cf.
(\ref{bed7}),
\begin{equation}
  \label{LL1}
  \frac{d}{dt} Q^{(0)}_{\alpha'' \alpha'} (t) =
  (A^\Delta_1)_{\alpha'' \alpha} Q^{(0)}_{\alpha \alpha'} (t),
  \qquad \alpha'' > \alpha,
\end{equation}
where $(A^\Delta_1)_{\alpha'' \alpha}$ is defined in (\ref{22A})
 and (\ref{ACmar}). Now we set $u_t = Q^{(0)}_{\alpha \alpha_0}(t)
 k_{\mu_0}$ and obtain from (\ref{LL}) and (\ref{LL1}), similarly as in
 (\ref{Ju1}),
\begin{gather*}
 u_t - k_t = \int_0^t Q^{(0)}_{\alpha \alpha_1}(t) \left( -
 B^\Delta_1\right)_{\alpha_1\alpha_2} k_s ds \geq 0,
\end{gather*}
which yields (\ref{Est}).

To prove that $r_t (\eta) := \vartheta^{-|\eta|}$, $t\geq 0$, is a
stationary solution we set
\[
k_t = Q_{\alpha \alpha_0} (t; B^\Delta_b) r_0,
\]
where $\alpha_0 > -\log \vartheta$ and $\alpha > \alpha_0$. Then the
following holds, cf. (\ref{Ju1}),
\[
k_t = r_t + \int_0^t Q_{\alpha \alpha_2} (t-s; B^\Delta_b)
L^\Delta_{\alpha_2 \alpha_1} r_s d s,
\]
where $\alpha_1 < \alpha_2$ are taken from $(\alpha_0, \alpha)$. For
the case considered, we have
\[
L^\Delta_{\alpha_2 \alpha_1} r_s = L^\Delta_{\alpha_2 \alpha_1} r_0
=0,
\]
which completes the proof for this case.

\section{The Proof of the Identification Lemma}

\label{S5}

To prove Lemma \ref{Wolflm} we use Proposition \ref{rhopn}. Note
that the solution mentioned in Lemma \ref{B1lm} already has
properties (ii) and (iii) of (\ref{9AZ}), cf. (\ref{z30}). Thus, it
remains to prove that also (i) holds. We do this as follows. First,
we approximate the evolution $k_0 \mapsto k_t$ established in Lemma
\ref{B1lm} by evolutions $k_{0,{\rm app}} \mapsto k_{t,{\rm app}}$
such that $k_{t,{\rm app}}$ has property (i). Then we prove that for
each $G\in B^{\star}_{\rm bs}(\Gamma_0)$, $\langle \!\langle G,
k_{t,{\rm app}} \rangle \! \rangle \to \langle \!\langle G, k_{t}
\rangle \! \rangle$ as the approximations are eliminated. The
limiting transition is based on the representation $\langle
\!\langle G, k_{t,{\rm app}} \rangle \! \rangle  = \langle \!\langle
G_t, k_{0,{\rm app}} \rangle \! \rangle$ in which we use the so
called predual evolution $G \mapsto G_t$. Then we just show that
$\langle \!\langle G_t, k_{0,{\rm app}} \rangle \! \rangle \to
\langle \!\langle G_t, k_{0} \rangle \! \rangle$.

\subsection{The predual evolution}

\label{SS5.1}

The aim of this subsection is to construct the evolution $B_{\rm
loc}(\Gamma_0) \ni G_0 \mapsto G_t \in \mathcal{G}_{\alpha_1}$, see
(\ref{z3}) and (\ref{NoR}), such that, for each $\alpha> \alpha_1$
and $k_0\in \mathcal{K}_{\alpha_1}$, the following holds, cf.
(\ref{fra}),
\begin{equation}
 \label{Jun}
\langle \! \langle G_0 , Q_{\alpha \alpha_1}(t;B^\Delta_b)k_0 \rangle \! \rangle =
\langle \! \langle G_t , k_0 \rangle \! \rangle,
\end{equation}
where $b\geq 0$ and $B^\Delta_b$ are as in (\ref{theta}) and
(\ref{key4}), respectively. Let us define the action of $B_b$ on
appropriate $G:\Gamma_0 \to \mathbb{R}$ via the duality
\[
\langle \! \langle G, B^\Delta_b k
\rangle \! \rangle =
\langle \! \langle B_b G,  k
\rangle \! \rangle.
\]
Similarly as in (\ref{key4}) we then get
\begin{gather}
 \label{Aqq4}
 (B_b G)(\eta)= b|\eta| G(\eta) +\int_{\mathbb{R}^d} \sum_{x\in
 \eta} a^{+} (x-y) G(\eta\setminus x \cup y) d y \\[.2cm] -  \sum_{x\in \eta}
 E^{-}(x,\eta \setminus x) G(\eta \setminus x).\nonumber
\end{gather}
For $\alpha_2 > \alpha_1$, let
$(B_b)_{\alpha_1 \alpha_2}$ be the bounded linear operator from
$\mathcal{G}_{\alpha_2}$ to $\mathcal{G}_{\alpha_1}$ the action of
which is defined in (\ref{Aqq4}).
As in estimating the
norm of $B^\Delta_b$ in (\ref{bed4}) one then gets
\begin{equation}
  \label{Cc2}
\|(B_b)_{\alpha_1\alpha_2}\| \leq \frac{\langle a^+ \rangle + b+ \langle
a^- \rangle e^{\alpha_2}}{e (\alpha_2 -\alpha_1)}.
\end{equation}
For the same $\alpha_2$ and $\alpha_1$,
let $S_{\alpha_1 \alpha_2}(t)$
be the restriction to $\mathcal{G}_{\alpha_2}$ of
the corresponding element of the semigroup mentioned in Lemma
\ref{1lm}. Then $S_{\alpha_1 \alpha_2}(t)$ acts as a bounded
contraction from $\mathcal{G}_{\alpha_2}$ to
$\mathcal{G}_{\alpha_1}$.

Now for a given $l\in \mathbb{N}$ and $\alpha$,  $\alpha_1$ as in
(\ref{Jun}), let $\delta$ and $\alpha^s$, $s=0, \dots, 2l+1$, be as
in (\ref{z40}). Then for $t>0$ and $(t,t_1 , \dots , t_l)\in
\mathcal{T}_l$, see (\ref{cone}), we define, cf. (\ref{Iwo}),
\begin{gather*}
\Omega_{\alpha_1\alpha}^{(l)} (t, t_1, \dots , t_l) = S_{\alpha_1\alpha^1}(t_l)
(B_b)_{\alpha^1\alpha^2} S_{\alpha^2 \alpha^3} (t_{l-1}- t_l) \times
\cdots \times \\[.2cm]
\times (B_b)_{\alpha^{2s-1} \alpha^{2s}} S_{\alpha^{2s}\alpha^{2s+1}}
(t_{l-s} - t_{l-s+1}) \cdots  (B_b)_{\alpha^{2l-1} \alpha^{2l}}
S_{\alpha^{2l}\alpha} (t - t_{1}). \nonumber
\end{gather*}
As in Proposition \ref{N1pn}, one shows that the map
\[
\mathcal{T}_l \ni  (t,t_1 , \dots , t_l) \mapsto \Omega_{\alpha_1\alpha}^{(l)} (t, t_1, \dots , t_l) \in
\mathcal{L}(\mathcal{G}_{\alpha},\mathcal{G}_{\alpha_1})
\]
is continuous.
Define
\begin{gather}
 \label{Qq}
 H^{(n)}_{\alpha_1\alpha} (t) = S_{\alpha_1\alpha}(t)
 + \sum_{l=1}^n \int_0^t
\int_0^{t_1} \cdots \int_0^{t_{l-1}} \Omega_{\alpha_1\alpha}^{(l)} (t, t_1 ,
\dots, t_l) d t_l \cdots d t_1.
\end{gather}
\begin{lemma}
\label{Qq1lm} For each $T\in (0, T(\alpha, \alpha_1;B^\Delta_b))$,
see (\ref{akko}) and (\ref{IwN1}), the sequence of operators defined
in (\ref{Qq}) converges in
$\mathcal{L}(\mathcal{G}_{\alpha},\mathcal{G}_{\alpha_1})$ to a
certain $H_{\alpha_1\alpha} (t)$ uniformly on $[0,T]$, and for each
$G_0 \in \mathcal{G}_\alpha$ and $k_0 \in \mathcal{K}_{\alpha_1}$ the
following holds
\begin{equation}
 \label{Qq1}
 \langle \! \langle G_0 , Q_{\alpha \alpha_1}(t;B^\Delta_b)k_0 \rangle \! \rangle =
\langle \! \langle H_{\alpha_1\alpha} (t) G_0, k_0 \rangle \! \rangle, \qquad t\in [0,T].
\end{equation}
\end{lemma}
\begin{proof}
For the operators defined in (\ref{Qq}),  similarly as in (\ref{Iv2N}) we get the following estimate
\[
 \| H^{(n)}_{\alpha_1\alpha} (t) - H^{(n-1)}_{\alpha_1\alpha} (t) \| \leq \frac{1}{n!} \left( \frac{n}{e}\right)^n
 \left( \frac{T}{T_\delta}\right)^n,
\]
which yields the convergence stated in the lemma. By direct inspection one gets that
\[
 \langle \! \langle G_0 , Q^{(n)}_{\alpha \alpha_1}(t;B^\Delta)k_0 \rangle \! \rangle =
\langle \! \langle H^{(n)}_{\alpha_1\alpha} (t) G_0, k_0 \rangle \! \rangle,
\]
see (\ref{KkN}). Then (\ref{Qq1}) is obtained from the latter in the
limit $n\to +\infty$. Similarly as in (\ref{Herf}),  for the limiting operator
the following estimate holds
\begin{equation}
\label{qx} \|H_{\alpha_1\alpha}(t)\| \leq \frac{T(\alpha,
\alpha_1;B^\Delta_b)}{T(\alpha, \alpha_1;B^\Delta_b) - t}.
\end{equation}
\end{proof}

\subsection{An auxiliary model}

\label{SS5.2}

The approximations mentioned at the beginning of this section employ
also an auxiliary model, which we introduce and study now. For this
model, we construct three kinds of evolutions. The first one is
$k_{0} \mapsto k_{t}\in \mathcal{K}_\alpha$ obtained as in Lemma
\ref{B1lm}. Another evolution $q_{0} \mapsto q_{t}\in
\mathcal{G}_\omega$ is constructed in such a way that $q_t$ is
positive definite in the sense that $\langle \! \langle G, q_t
\rangle \! \rangle \geq 0$ for all $G\in B^\star_{\rm
bs}(\Gamma_0)$. These evolutions, however, take place in different spaces.
To relate them to each other  we construct one more evolution, $u_0
\mapsto u_t$, which takes place in the intersection of the mentioned Banach
spaces. The aim is to show that $k_{t} = u_t = q_t$ and thereby to get
the desired property of $k_{t}$. Thereafter, we prove the
convergence mentioned above.

\subsubsection{The model}

The function
\begin{equation}
  \label{Qq20}
  \varphi_\sigma (x) = \exp\left(- \sigma |x|^2 \right), \qquad \sigma >0, \quad x\in
  \mathbb{R}^d,
\end{equation}
has the following evident properties
\begin{equation}
  \label{Be}
\bar{\varphi}_\sigma := \int_{\mathbb{R}^d} \varphi (x) d x <
\infty, \quad \quad  \varphi_\sigma (x) \leq 1, \ \ \ x\in
\mathbb{R}^d.
\end{equation}
The model we need is characterized by $L$ as in (\ref{R20}) with
$E^{+}(x,\eta)$, cf. (\ref{Ra20}), replaced by
\begin{equation}
 \label{Be1}
  E^{+}_\sigma (x, \eta) = \varphi_\sigma(x)E^{+}_\sigma (x, \eta) = \varphi_\sigma(x) \sum_{y\in \eta} a^{+} (x-y).
\end{equation}

\subsubsection{The evolution in $\mathcal{K}_\alpha$}

For the new model (with $E^{+}_\sigma$ as in (\ref{Be1})), the operator $
L^{\Delta,\sigma}$ corresponding  to $L^\Delta$ takes the form, cf.
(\ref{21A}) -- (\ref{23A}) and (\ref{key3}) -- (\ref{key5}),
\begin{equation}
  \label{Qq21}
 L^{\Delta,\sigma} = A^{\Delta,\sigma} + B^{\Delta,\sigma} = A^{\Delta,\sigma}_b + B^{\Delta,\sigma}_b.
\end{equation}
Here
\begin{eqnarray}
  \label{Qq22}
A^{\Delta,\sigma} & = & A^{\Delta}_1 + A^{\Delta,\sigma}_2, \qquad
A^{\Delta,\sigma}_b  =  A^{\Delta}_{1,b} + A^{\Delta,\sigma}_2,  \\[.2cm]
B^{\Delta,\sigma}& = & B^{\Delta}_1 + B^{\Delta,\sigma}_2, \qquad
B^{\Delta,\sigma}_b =  B^{\Delta}_1 + B^{\Delta,\sigma}_{2,b},
\nonumber
\end{eqnarray}
where $A^{\Delta}_1$,  $B^{\Delta}_1$, and $A^{\Delta}_{1,b}$ are the same as in
(\ref{22A}), (\ref{23A}), and (\ref{key3}), respectively,
and
\begin{eqnarray}
  \label{Qq23}
\left(A^{\Delta,\sigma}_2 k\right)(\eta) & = & \sum_{x\in \eta}
\varphi_\sigma(x) E^{+} (x, \eta\setminus x) k(\eta \setminus x),
\\[.2cm]
\left(B^{\Delta,\sigma}_2 k\right)(\eta) & = & b|\eta| k(\eta) + \int_{\mathbb{R}^d}
\sum_{x\in \eta} \varphi_\sigma(x) a^{+} (x-y) k(\eta \setminus
x\cup y) dy. \nonumber
\end{eqnarray}
Note that these $A^{\Delta,\sigma}_b$ and $B^{\Delta,\sigma}_b$
define the corresponding bounded operators acting from
$\mathcal{K}_{\alpha'}$ to $\mathcal{K}_{\alpha}$ for each real
$\alpha > \alpha'$. As in (\ref{Iw1}) we then set
\begin{equation}
  \label{QQ1}
  \mathcal{D}^{\Delta, \sigma}_\alpha = \{k\in
  \mathcal{K}_\alpha: L^{\Delta, \sigma} k \in \mathcal{K}_\alpha\},
\end{equation}
and thus define the corresponding operator $(L^{\Delta,
\sigma}_\alpha, \mathcal{D}^{\Delta, \sigma}_\alpha)$. Along with
(\ref{Cauchy}) we also consider
\begin{equation}
  \label{Cauchy-sigma}
 \frac{d}{dt} k_t = L^{\Delta, \sigma}_\alpha k_t, \qquad
 k_t|_{t=0}= k_0 \in \mathcal{D}^{\Delta, \sigma}_\alpha.
\end{equation}
By the literal repetition of the construction used in the proof of Lemma
\ref{IN1lm} one obtains the operators $Q^\sigma_{\alpha \alpha'}(t; B^{\Delta, \sigma}_b)$,
$(\alpha, \alpha', t) \in \mathcal{A}(B^\Delta_b)$, see (\ref{aKK}), the norm of which satisfies,
cf. (\ref{Herf}),
\begin{equation}
  \label{QQ3}
  \|Q^\sigma_{\alpha \alpha'}(t; B^{\Delta, \sigma}_b)\| \leq
  \frac{T(\alpha, \alpha';B^\Delta_b)}{T(\alpha, \alpha';B^\Delta_b) -
  t},
\end{equation}
which is uniform in $\sigma$.
\begin{lemma}
  \label{sQQlm}
Let $\alpha_1$ and $\alpha_2$ be as in Lemma \ref{B1lm}. Then for a
given $k_0\in \mathcal{K}_{\alpha_1}$, the unique solution of
(\ref{Cauchy-sigma}) in $\mathcal{K}_{\alpha_2}$ is given by
\begin{equation}
  \label{QQ2}
  k_t = Q^\sigma_{\alpha \alpha_1}(t; B^{\Delta, \sigma}_b)k_0, \qquad \alpha \in (\alpha_1, \alpha_2), \ \
  t< T(\alpha_2, \alpha_1;B^\Delta_b).
\end{equation}
\end{lemma}
\begin{proof}
Repeat the proof of Lemma \ref{B1lm}.
\end{proof}

\subsubsection{The evolution in $\mathcal{U}_{\sigma,\alpha}$}
For $\varphi_\sigma$ as in (\ref{Qq20}) we set
\begin{equation*}
  e(\varphi_\sigma ; \eta) = \prod_{x\in \eta} \varphi_\sigma (x),
  \qquad \eta \in \Gamma_0,
\end{equation*}
and introduce the following Banach space. For $u:\Gamma_0 \to
\mathbb{R}$, we define the norm, cf. (\ref{z30}),
\begin{equation}
  \label{Qq31}
\|u\|_{\sigma, \alpha} = \esssup_{\eta \in \Gamma_0}
\frac{|u(\eta)|\exp(-\alpha |\eta|)}{e(\varphi_\sigma ; \eta)}.
\end{equation}
Thereafter, set
\begin{equation*}
\mathcal{U}_{\sigma, \alpha} = \{u:\Gamma_0 \to \mathbb{R}:
\|u\|_{\sigma, \alpha} < \infty\}.
\end{equation*}
By (\ref{Qq20}) and (\ref{z30})  we have that
\begin{equation*}
\|u\|_{ \alpha} \leq \|u\|_{\sigma, \alpha}, \qquad u \in
\mathcal{U}_{\sigma, \alpha},
\end{equation*}
which yields $\mathcal{U}_{\sigma, \alpha} \hookrightarrow
\mathcal{K}_{\alpha}$. Moreover, as in (\ref{z31}) we also have that
$\mathcal{U}_{\sigma, \alpha'} \hookrightarrow \mathcal{U}_{\sigma,
\alpha}$ for each real $\alpha > \alpha'$.

Now let us define the operator $L^{\Delta,\sigma}_{\alpha,u}$ in
$\mathcal{U}_{\sigma,\alpha}$ the action of which is described in (\ref{Qq21}) --
(\ref{Qq23}) and the domain is, cf. (\ref{QQ1}),
\begin{equation}
  \label{QQ4}
  \mathcal{D}^{\Delta,\sigma}_{\alpha,u} = \{ u \in
  \mathcal{U}_{\sigma,\alpha}: L^{\Delta,\sigma}u\in
  \mathcal{U}_{\sigma,\alpha}\}.
\end{equation}
Then we consider
\begin{equation}
  \label{Cauchy-u}
 \frac{d}{dt} u_t = L^{\Delta,\sigma}_{\alpha,u} u_t, \qquad
 u_t|_{t=0} = u_0 \in  \mathcal{D}^{\Delta,\sigma}_{\alpha,u}.
\end{equation}
Note that $\mathcal{U}_{\sigma,\alpha''} \subset
\mathcal{D}(L^{\Delta,\sigma}_{\alpha,u})$ for each $\alpha''
<\alpha$, and
\begin{equation}
  \label{QQ5}
(L^{\Delta,\sigma}_{\alpha,u},
\mathcal{D}^{\Delta,\sigma}_{\alpha,u}) \subset
(L^{\Delta,\sigma}_{\alpha},
\mathcal{D}^{\Delta,\sigma}_{\alpha}).
\end{equation}
Our aim now is to prove that the problem (\ref{Cauchy-u}) with $u_0
\in \mathcal{U}_{\sigma, \alpha_1}$ has a unique solution in
$\mathcal{U}_{\sigma, \alpha_2}$, where $\alpha_1 < \alpha_2$ are as
in Lemma \ref{B1lm}. To this end we first construct the semigroup
analogous to that obtained in Lemma \ref{1lm}. Thus, in the predual
space $\mathcal{G}_{\sigma, \alpha}$ equipped with the norm, cf.
(\ref{NoR}),
\begin{equation*}
|G|_{\sigma,\alpha}:= \int_{\Gamma_0} |G(\eta)| \exp(\alpha|\eta|)
e(\varphi_\sigma;\eta) \lambda ( d\eta)
\end{equation*}
we define the action of $A^\sigma_b$ as follows, cf. (\ref{z1}),
\begin{eqnarray*}
A^\sigma_b & = & A_{1,b} + A^\sigma_2\\[.2cm]
(A^\sigma_2 G)(\eta)& = & \int_{\mathbb{R}^d} \varphi_\sigma
(y)E^{+} (y,\eta) G(\eta\cup y) d y, \nonumber
\end{eqnarray*}
and $A_{1,b}$ acts as in (\ref{z1}). Then we have, cf. (\ref{z7}),
\begin{eqnarray*}
& & |A^\sigma_2 G|_{\sigma,\alpha} \\[.2cm] & & \quad \leq \int_{\Gamma_0} \left(
\int_{\mathbb{R}^d}  \varphi_\sigma (y)E^{+} (y,\eta) |G(\eta\cup
y)| d y \right) \exp(\alpha|\eta|) e(\varphi_\sigma;\eta) \lambda (
d\eta) \nonumber  \\[.2cm]& &\quad =  \int_{\Gamma_0}e^{-\alpha}
\left(\sum_{x\in \eta} E^{+} (x, \eta\setminus x)\right) |G(\eta)|
\exp(\alpha|\eta|) e(\varphi_\sigma;\eta) \lambda (d\eta) \nonumber
\\[.2cm]
& & \quad \leq  (e^{-\alpha}/\vartheta) |A_{1,b} G|_{\sigma,\alpha}. \nonumber
\end{eqnarray*}
Now the existence of the substochastic semigroup $\{S_{\sigma,
\alpha}(t)\}_{t\geq0}$ generated by $(A^\sigma_b,
\mathcal{D}_{\sigma,\alpha})$ follows as in Lemma \ref{1lm}. Here,
cf. (\ref{IW1}),
\[
\mathcal{D}_{\sigma,\alpha} := \{G \in \mathcal{G}_{\sigma,\alpha}:
E_b(\cdot) G \in \mathcal{G}_{\sigma,\alpha}\}.
\]
Let $S^{\odot}_{\sigma, \alpha}(t)$ be the sun-dual to $S_{\sigma,
\alpha}(t)$, cf. (\ref{ACa}). Then for each $\alpha' <\alpha$ and
any $u\in \mathcal{U}_{\sigma, \alpha'}$, the map $$[0, +\infty)\ni
t \mapsto S^{\odot}_{\sigma, \alpha}(t)u \in \mathcal{U}_{\sigma,
\alpha}$$ is continuous, see Proposition \ref{1pn}. For real
$\alpha'< \alpha$ and $t>0$, let
$\Sigma^{\sigma,u}_{\alpha\alpha'}(t)$ be the restriction of
$S^{\odot}_{\sigma, \alpha}(t)$ to $\mathcal{U}_{\sigma, \alpha'}$.
Then the map
\[
[0, +\infty) \ni t \mapsto \Sigma^{\sigma,u}_{\alpha\alpha'}(t) \in
\mathcal{L}(\mathcal{U}_{\sigma,\alpha'},
\mathcal{U}_{\sigma,\alpha})
\]
is continuous and such that, cf. (\ref{norma}),
\begin{equation}
  \label{QQ9}
  \|\Sigma^{\sigma,u}_{\alpha\alpha'}(t) \| \leq 1, \qquad t\geq
  0.
\end{equation}
Now we define $(B^{\Delta,\sigma}_b)_{\alpha\alpha'}$ which acts
from $\mathcal{U}_{\sigma,\alpha'}$ to $\mathcal{U}_{\sigma,\alpha}$
according to (\ref{Qq22}) and (\ref{Qq23}). Then its norm satisfies
\begin{equation}
  \label{QQ10}
\|(B^{\Delta,\sigma}_b)_{\alpha\alpha'}\| \leq \frac{\langle a^{+}
\rangle + b+ \langle a^{-} \rangle e^{\alpha}}{e(\alpha - \alpha')}.
\end{equation}
In proving this we take into account that $\varphi_\sigma (x) \leq
1$ and repeat the arguments used in obtaining (\ref{bed4}).

For real $\alpha_2 > \alpha_1 > - \log \vartheta$, we take
$\alpha \in (\alpha_1, \alpha_2]$ and then pick  $\delta < \alpha - \alpha_1$
as in the proof of Lemma \ref{IN1lm}. Next, for
$l\in \mathbb{N}$ we
divide $[\alpha_1,\alpha]$ into
subintervals according to (\ref{z40}) and take $(t,t_1, \dots ,
t_l)\in \mathcal{T}_l$, see (\ref{cone}). Then define, cf.
(\ref{Iwo}),
\begin{eqnarray*}
& & \Pi^{l,\sigma}_{\alpha \alpha_1} (t,t_1, \dots , t_l) =
\Sigma^{\sigma,u}_{\alpha\alpha^{2l}}(t-t_1)(B^{\Delta,\sigma}_b)_{\alpha^{2l}\alpha^{2l-1}} \\[.2cm]
& & \quad \times
\Sigma^{\sigma,u}_{\alpha^{2l-1}\alpha^{2l-2}}(t_1-t_2) \times
\cdots \times
\Sigma^{\sigma,u}_{\alpha^3\alpha^{2}}(t_{l-1}-t_l)(B^{\Delta,\sigma}_b)_{\alpha^{2}\alpha^{1}}
\Sigma^{\sigma,u}_{\alpha^1\alpha_1}(t_{l}). \nonumber
\end{eqnarray*}
Thereafter, for $n\in \mathbb{N}$ we set, cf. (\ref{KkN}),
\begin{eqnarray*}
& & U^{(n)}_{\alpha\alpha_1}(t) =
\Sigma^{\sigma,u}_{\alpha\alpha_1}(t)
\\[.2cm] & & \quad + \sum_{l=1}^n \int_0^t \int_0^{t_1} \cdots
\int_0^{t_{n-1}} \Pi^{l,\sigma}_{\alpha \alpha_1} (t,t_1, \dots ,
t_l) d t_l \cdots d t_1. \nonumber
\end{eqnarray*}
By means of (\ref{QQ9}) and (\ref{QQ10}) we then prove that the
sequence $\{U^{(n)}_{\alpha\alpha_1}(t)\}_{n\in \mathbb{N}}$
converges in $\mathcal{L}(\mathcal{U}_{\sigma,\alpha_1},
\mathcal{U}_{\sigma,\alpha})$, uniformly on $[0,T]$, $T< T(\alpha,
\alpha_1;B^\Delta_b)$, see (\ref{akko}) and (\ref{IwN1}). The limit $
U_{\alpha\alpha_1}(t) \in \mathcal{L}(\mathcal{U}_{\sigma,\alpha_1},
\mathcal{U}_{\sigma,\alpha})$ has the property, cf. (\ref{bed7}),
\begin{equation*}
\frac{d}{dt} U_{\alpha_2\alpha_1}(t) = \left(
(A^{\Delta,\sigma}_b)_{\alpha_2\alpha} +
(B^{\Delta,\sigma}_b)_{\alpha_2\alpha} \right)U_{\alpha\alpha_1}(t),
\end{equation*}
where $(A^{\Delta,\sigma}_b)_{\alpha_2\alpha}\in
\mathcal{L}(\mathcal{U}_{\sigma,\alpha},
\mathcal{U}_{\sigma,\alpha_2})$ is defined in (\ref{Qq22}) and
(\ref{Qq23}), analogously to (\ref{QQ10}). Note that
\begin{equation}
  \label{QQ14}
\forall u \in \mathcal{U}_{\sigma, \alpha} \quad
L^{\Delta,\sigma}_{\alpha_2,u} u = \left(
(A^{\Delta,\sigma}_b)_{\alpha_2\alpha} +
(B^{\Delta,\sigma}_b)_{\alpha_2\alpha} \right) u,
\end{equation}
see (\ref{QQ4}). Now we can state the following analog of Lemma
\ref{B1lm}.
\begin{lemma}
  \label{uQQlm}
Let $\alpha_2 > \alpha_1 > - \log \vartheta$ be as in Lemma
\ref{B1lm}. Then the problem (\ref{Cauchy-u}) with $u_0\in
\mathcal{U}_{\sigma, \alpha_1}$ has a unique solution $u_t \in
\mathcal{U}_{\sigma, \alpha_2}$ on the time interval $[0,
T(\alpha_2, \alpha_1;B^\Delta_b))$.
 \end{lemma}
\begin{proof}
Fix $T<T(\alpha_2, \alpha_1;B^\Delta_b)$ and find $\alpha\in
(\alpha_1, \alpha_2)$ such that also $T<T(\alpha',
\alpha_1;B^\Delta_b)$. Then, cf. (\ref{fra}),
\begin{equation}
  \label{QQ15}
 u_t := U_{\alpha \alpha_1}(t) u_0
\end{equation}
is the solution in question, which can be checked by means of
(\ref{QQ14}). Its uniqueness can be proved by the literal repetition of
the corresponding arguments used in the proof of Lemma \ref{B1lm}.
\end{proof}
\begin{corollary}
  \label{uQQco}
Let $k_t$ be the solution of the problem (\ref{Cauchy-sigma}) with
$k_0 \in \mathcal{U}_{\sigma,\alpha_1}$ mentioned in Lemma
\ref{sQQlm}. Then $k_t$ coincides with the solution mentioned in
Lemma \ref{uQQlm}.
\end{corollary}
\begin{proof}
Since $(L^{\Delta,\sigma}_{\alpha},
\mathcal{D}^{\Delta,\sigma}_{\alpha})$ is an extension of
$(L^{\Delta,\sigma}_{\alpha,u},
\mathcal{D}^{\Delta,\sigma}_{\alpha,u})$, see (\ref{QQ5}), and the
embedding $\mathcal{U}_{\sigma,\alpha}\hookrightarrow
\mathcal{K}_\alpha$ is continuous, the solution as in (\ref{QQ15})
with $u_0= k_0$ satisfies also (\ref{Cauchy-sigma}), and hence
coincides with $k_t$ in view of the uniqueness stated in Lemma
\ref{sQQlm}.
\end{proof}

\subsubsection{The evolution in $\mathcal{G}_{\omega}$}
We recall that the space $\mathcal{G}_{\alpha}$ was introduced in
(\ref{z3}), (\ref{NoR}), where we used it as a predual space to
$\mathcal{K}_{\alpha}$. Now we employ
$\mathcal{G}_{\alpha}$ to get the positive definiteness mentioned at
the beginning of this subsection. Here, however, we write
$\mathcal{G}_{\omega}$ to show that we use it not as a predual
space.

Let $L^{\Delta, \sigma}$ be as in (\ref{Qq21}). For
$\omega\in \mathbb{R}$, we set, cf. (\ref{QQ1}) and (\ref{QQ4}),
\begin{equation*}
\mathcal{D}^{\Delta,\sigma}_\omega = \{ q \in \mathcal{G}_\omega:
L^{\Delta,\sigma}q \in \mathcal{G}_\omega\}.
\end{equation*}
Then we define the corresponding operator
$(L^{\Delta,\sigma}_\omega, \mathcal{D}^{\Delta,\sigma}_\omega)$ and
consider the following Cauchy problem
\begin{equation}
 \label{Cauchy-omega}
 \frac{d}{dt} q_t = L^{\Delta,\sigma}_\omega q_t , \qquad q_t|_{t=0} = q_0 \in \mathcal{D}^{\Delta,\sigma}_\omega.
\end{equation}
As above, one can show that $\mathcal{G}_{\omega'} \subset \mathcal{D}^{\Delta,\sigma}_\omega$ for each $\omega' > \omega$.
By (\ref{Qq31}) and (\ref{NoR}) for $u\in \mathcal{U}_{\sigma,
\alpha}$ we have
\begin{eqnarray}
 \label{QQ18}
 |u|_\omega & \leq & \|u\|_{\sigma,\alpha} \int_{\Gamma_0} \exp((\omega + \alpha)|\eta|)e(\varphi_\sigma;\eta) \lambda (d\eta)\\[.2cm]
& \leq & \|u\|_{\sigma,\alpha} \exp\left( \bar{\varphi}_\sigma
e^{\omega + \alpha}\right), \nonumber
\end{eqnarray}
see also (\ref{Be}). Hence $\mathcal{U}_{\sigma, \alpha}
\hookrightarrow \mathcal{G}_\omega$ for each $\omega$ and $\alpha$.
Like in (\ref{QQ5}) we then get
\begin{equation}
  \label{QQ19}
(L^{\Delta,\sigma}_{\alpha,u},
\mathcal{D}^{\Delta,\sigma}_{\alpha,u}) \subset
(L^{\Delta,\sigma}_{\omega}, \mathcal{D}^{\Delta,\sigma}_{\omega}).
\end{equation}
\begin{lemma}
 \label{qQQlm}
Assume that the problem (\ref{Cauchy-omega}) with $\omega > 0$ and
$q_0 \in \mathcal{G}_{\omega'}$, $\omega'> \omega$, has a solution,
$q_t\in \mathcal{G}_\omega$, on some time interval $[0, T(\omega',
\omega))$. Then this solution is unique.
\end{lemma}
\begin{proof}
 Set
\begin{equation*}
   w_t(\eta) = (-1)^{|\eta|} q_t (\eta),
\end{equation*}
which is an isometry on $\mathcal{G}_\omega$. Then $q_t$ solves
(\ref{Cauchy-omega}) if and only if $w_t$ solves the following
equation
\begin{gather}
  \label{Qq14}
\frac{d}{dt} w_t(\eta) = - E(\eta) w_t(\eta) + \int_{\mathbb{R}^d}
E^{-}(y,\eta) w_t(\eta\cup y)d y\\[.2cm]
- \sum_{x\in \eta} \varphi_\sigma(x)  E^{+} (x, \eta \setminus x)
w_t (\eta \setminus
x) \nonumber \\[.2cm] + \int_{\mathbb{R}^d}\sum_{x\in \eta}  \varphi_\sigma(x)  a^{+} (x-y) w_t (x\setminus
x \cup y) dy. \nonumber
\end{gather}
Set
\[
 \mathcal{D}_\omega =\{ w \in \mathcal{G}_\omega : E(\cdot) w \in \mathcal{G}_\omega\}.
\]
By Proposition \ref{le:substoch} we prove that the operator defined
by the  first two summands in (\ref{Qq14}) with domain $
\mathcal{D}_\omega$ generates a substochastic semigroup,
$\{V_\omega(t)\}_{t\geq 0}$, acting in $\mathcal{G}_\omega$. Indeed, in this
case the condition analogous to that in (\ref{cond:substoch}) takes
the form, cf. (\ref{Qq15}),
\begin{eqnarray*}
& - & \int_{\Gamma_0} E(\eta) w(\eta) \exp(\omega |\eta|) \lambda (
d\eta)\\[.2cm] & & \qquad + r^{-1}e^{-\omega} \int_{\Gamma_0} E^{-}(\eta) w(\eta)
\exp(\omega |\eta|) \lambda ( d\eta) \leq 0,
\end{eqnarray*}
which certainly holds for each $\omega>0$ and an appropriate $r<1$.
For each $\omega'' \in (0, \omega)$,  we have that
$\mathcal{G}_\omega \hookrightarrow \mathcal{G}_{\omega''}$, and the
second two  summands in (\ref{Qq14}) define a bounded operator,
$W_{\omega''\omega}: \mathcal{G}_\omega \to \mathcal{G}_{\omega''}$,
the norm of which can be estimated as follows, cf. (\ref{Cc2}),
\begin{equation}
  \label{Qq16}
  \|W_{\omega''\omega} \| \leq \frac{(e^{\omega} +1) \langle a^{+}
  \rangle}{e(\omega - \omega'')}.
\end{equation}
Assume now that (\ref{Qq14}) has two solutions corresponding to the
same initial condition $w_0 (\eta) = (-1)^{|\eta|} q_0(\eta)$. Let
$v_t$ be their difference. Then it solves the following equation,
cf. (\ref{fra1}),
\begin{equation}
  \label{Qq17}
 v_t = \int_0^t V_{\omega''}(t-s) W_{\omega''\omega} v_s d s,
\end{equation}
where $v_t$ on the left-hand side is considered as an element of
$\mathcal{G}_{\omega''}$ and $t>0$ will be chosen later. Now for a
given $n\in \mathbb{N}$, we set $\epsilon = (\omega - \omega'')/n$
and then $\omega^l := \omega - l \epsilon$, $l=0, \dots , n$. Thereafter, we
iterate (\ref{Qq17}) and get
\begin{eqnarray*}
v_t & = & \int_0^t \int_0^{t_1} \cdots \int_0^{t_{n-1}}
V_{\omega''}(t-t_1) W_{\omega''\omega^{n-1}} V_{\omega^{n-1}}
(t_1-t_2) \times \cdots \times \\[.2cm] & \times &
W_{\omega^{2}\omega^{1}} V_{\omega^1} (t_{n-1} - t_n) W_{\omega^{1}
\omega} v_{t_n} d t_n \cdots d t_1.
\end{eqnarray*}
Similarly as in (\ref{fraa1}), by (\ref{Qq16}) this yields the
following estimate
\begin{eqnarray*}
|v_t|_{\omega''} \leq \frac{1}{n!} \left( \frac{n}{e}\right)^n
\left( \frac{t\langle a^{+} \rangle (e^{\omega} +1)}{\omega -
\omega''}\right)^n \sup_{s\in [0,t]} |v_s|_{\omega}.
\end{eqnarray*}
The latter implies that $v_t =0$ for $t < (\omega - \omega'')/
\langle a^{+} \rangle (e^{\omega} +1)$. To prove that $v_t =0$ for
all $t$ of interest one has to repeat the above procedure
appropriate number of times.
\end{proof}
\vskip.1cm Recall that each $\mathcal{U}_{\sigma, \alpha}$ is
continuously embedded into each $\mathcal{G}_\omega$, see
(\ref{QQ18}).
\begin{corollary}
 \label{qQQco}
For each $\omega>0$, the problem (\ref{Cauchy-omega}) with $q_0 \in
\mathcal{U}_{\sigma, \alpha_0}$ has a unique solution $q_t$ which
coincides with the solution $u_t\in \mathcal{U}_{\sigma, \alpha}$
mentioned in Lemma \ref{uQQlm}.
\end{corollary}
\begin{proof}
By (\ref{QQ19}) $u_t$ is a solution of (\ref{Cauchy-omega}). Its
uniqueness follows by Lemma \ref{qQQlm}.
\end{proof}

\subsection{Local evolution}

\label{SS5.3}

In this subsection we pass to the so called local evolution of
states of the auxiliary model (\ref{Qq21}), (\ref{Qq22}). For this evolution,
the corresponding `correlation function'
 $q_t \in \mathcal{G}_\omega$ has the
positive definiteness in question. Then we apply Corollaries
\ref{uQQco} and \ref{qQQco} to get the same for the evolution in
$\mathcal{K}_\alpha$. Thereafter, we pass to the limit and get the
proof of Lemma \ref{Wolflm}.

\subsubsection{The evolution of densities}
In view of (\ref{Feb1}), each state with the property
$\mu(\Gamma_0)=1$ can be redefined as a probability measure on
$\mathcal{B}(\Gamma_0)$, cf. Remark \ref{Febrk1}. Then the
Fokker-Planck equation (\ref{R1}) can be studied directly, see
\cite[Eq. (2.8)]{K}. Its solvability is described in \cite[Theorem
2.2]{K}, which, in particular, states that the solution is
absolutely continuous with respect to the Lebesgue-Poisson measure
$\lambda$ if $\mu_0$ has the same property. In view of this we write
the corresponding problem for the density
\begin{equation}
  \label{Ap4}
  R_t := \frac{d \mu_t}{d \lambda},
\end{equation}
see also \cite[Eq. (2.16)]{K}, and obtain
\begin{equation}
  \label{Rr}
\frac{d}{dt} R_t (\eta)  = (L^{\dagger,\sigma} R_t)(\eta), \quad
R_t|_{t=0} = R_0,
\end{equation}
where
\begin{eqnarray}
  \label{Ap5}
(L^{\dagger,\sigma} R)(\eta) & := & - \Psi_\sigma (\eta) R (\eta) +
\sum_{x\in
  \eta} \varphi_\sigma (x) E^{+} (x, \eta \setminus x) R_t (\eta \setminus x) \qquad  \\[.2cm]
 & + & \int_{\mathbb{R}^d} \left( m + E^{-} (x, \eta)\right) R_t
(\eta\cup x) d x, \nonumber
\end{eqnarray}
and
\begin{equation*}
\Psi_\sigma (\eta)= E(\eta) + \int_{\mathbb{R}^d} \varphi_\sigma (x)
E^{+} (x, \eta) dx.
\end{equation*}
We solve (\ref{Rr}) in the Banach spaces $\mathcal{G}_0 =
L^1(\Gamma_0, d\lambda)$, cf. (\ref{z3}). For $n\in \mathbb{N}$ we
denote by
 $\mathcal{G}_{0,n}$ the subset of  $\mathcal{G}_0$ consisting of
 all those $R:\Gamma_0 \to \mathbb{R}$ for which
 \[
\int_{\mathbb{R}^d} |\eta|^n \left\vert R(\eta)\right\vert \lambda (
d\eta) < \infty.
 \]
Let also  $\mathcal{G}_\omega^+$ stand for the cone of positive
elements of  $\mathcal{G}_\omega$. Set
\begin{equation}
  \label{Ap6}
 \mathcal{D}_0 = \{ R\in \mathcal{G}_0: \Psi_\sigma R \in \mathcal{G}_0\}.
\end{equation}
Then the relevant part of \cite[Theorem 2.2]{K} can be formulated as
follows.
\begin{proposition}
  \label{Ap1pn}
The closure in $\mathcal{G}_0$ of the operator $(L^{\dagger, \sigma},
\mathcal{D}_0)$ defined in (\ref{Ap5}) and (\ref{Ap6}) generates a
stochastic semigroup $\{S^{\dagger,\sigma}(t)\}_{t\geq 0}:=
S^{\dagger,\sigma}$ of bounded operators in  $\mathcal{G}_0$, which
leaves invariant each $\mathcal{G}_{0,n}$, $n\in \mathbb{N}$.
Moreover, for each $\beta'>0$ and $\beta \in (0, \beta')$, $R\in
\mathcal{G}_{\beta'}^{ +}$ implies $S^{\dagger,\sigma} (t) R\in
\mathcal{G}_{\beta}^{ +}$ holding for all $t < T(\beta' , \beta)$,
where $T(\beta' , \beta)= +\infty $ for $\langle a^{+}\rangle = 0$,
and
\begin{equation}
 \label{Qq5}
 T(\beta' , \beta ) = (\beta' - \beta)e^{-\beta'}/\langle a^{+}\rangle, \qquad {\rm for} \ \ \langle a^{+}\rangle>0.
\end{equation}
\end{proposition}
Let now $\mu_0$ be the initial state as in Theorem \ref{1tm}. Then
for each $\Lambda \in \mathcal{B}_{\rm b}(\mathbb{R}^d)$, the
projection $\mu^\Lambda$ is absolutely continuous with respect to
$\lambda^\Lambda$, see (\ref{ProJ}). For this $\mu_0$, and for
$\Lambda\in \mathcal{B}_{\rm b}(\mathbb{R}^d)$ and $N \in
\mathbb{N}$, we set, see (\ref{Ap4}),
\begin{equation}
 \label{y8}
 R^{\Lambda}_0 (\eta) = \frac{d \mu^{\Lambda}}{ d
 \lambda^{\Lambda}}(\eta) \mathbb{I}_{\Gamma_{\Lambda}} (\eta), \qquad R^{\Lambda, N}_0 (\eta) = R^{\Lambda}_0 (\eta) I_{N}
 (\eta),    \ \ \eta \in \Gamma_0.
\end{equation}
Here $I_{N}$ and $\mathbb{I}_{\Gamma_\Lambda}$ are the indicator
functions of the sets $\{\eta \in \Gamma_0: |\eta|\leq N\}$, $N\in
\mathbb{N}$, and $\Gamma_\Lambda$, respectively. Clearly,
\begin{equation}
 \label{secure}
 \forall \beta >0 \qquad R^{\Lambda, N}_0\in \mathcal{G}^+_\beta.
\end{equation}
Set
\begin{equation}
 \label{y10}
 R^{\Lambda, N}_t = S^{\dagger,\sigma}(t) R^{\Lambda, N}_0 , \quad
 t>0,
\end{equation}
where $S^{\dagger,\sigma}$ is the semigroup as in Proposition
\ref{Ap1pn}. Then also $R^{\Lambda, N}_t \in \mathcal{G}^+_0$ for
all $t>0$.

For some $G\in B_{\rm bs}(\Gamma_0)$,  let us consider $F=KG$, cf.
(\ref{7A}). Since $G(\xi) =0$ for all $\xi$ such that $|\xi|>N(G)$,
see Definition \ref{Febdf}, we have $F\in \mathcal{F}_{\rm
cyl}(\Gamma)$ and
\begin{equation*}
 |F (\gamma)| \leq (1+|\gamma|)^{N(G)} C(G), \qquad \gamma\in
 \Gamma_0,
\end{equation*}
for some $ C(G)>0$. By Proposition \ref{Ap1pn} we then have from the
latter
\begin{equation}
  \label{qQ}
\left\vert \langle \! \langle KG , R^{\Lambda,N}_t \rangle \!
\rangle \right\vert < \infty.
\end{equation}

\subsubsection{The evolution of local correlation functions}

For a given $\mu \in \mathcal{P}_{\rm sP}$, the correlation function
$k_\mu$ and the local densities $R^\Lambda_\mu$, $\Lambda \in
\mathcal{B}_{\rm b}(\mathbb{R}^d)$, see (\ref{RN}), are related to
each other by (\ref{9AA}). In the first formula of (\ref{y8}) we
extend $R_0^\Lambda$ to the whole $\Gamma_0$. Then the corresponding
integral as in (\ref{9AA}) coincides with $k_{\mu_0}$ only on
$\Gamma_\Lambda$. The truncation made in the second formula in
(\ref{y8})  diminishes $R_0^\Lambda$. Its aim is to
satisfy (\ref{secure}).
Thus, with a
certain abuse of the terminology we call
\begin{equation}
 \label{y172}
q^{\Lambda,N}_0 (\eta) = \int_{\Gamma_0} R^{\Lambda,N}_0 (\eta \cup
\xi) \lambda(d \xi)
\end{equation}
{\it local correlation function}. The evolution $q^{\Lambda, N}_0
\mapsto q^{\Lambda, N}_t$ can be obtained from (\ref{y10}) by
setting
\begin{equation}
  \label{y13}
q^{\Lambda,N}_t (\eta) = \int_{\Gamma_0} R^{\Lambda,N}_t (\eta \cup
\xi) \lambda (d\xi), \qquad t\geq 0.
\end{equation}
However, so far this can only be used in a weak sense based on
(\ref{qQ}). Note that for $G\in B^{\star}_{\rm bs}(\Gamma_0)$, cf.
(\ref{9AY}), we have
\begin{equation}
  \label{Qq10}
 \langle \! \langle G , q^{\Lambda,N}_t  \rangle \! \rangle  =
\langle \! \langle KG , R^{\Lambda,N}_t \rangle \! \rangle \geq 0,
\end{equation}
since $R^{\Lambda,N}_t \in \mathcal{G}^+_0$. To place the evolution
$q^{\Lambda, N}_0 \mapsto q^{\Lambda, N}_t$ into an appropriate
Banach space we use the concluding part of Proposition \ref{Ap1pn}
and the following fact
\begin{equation}
  \label{Qq11}
\int_{\Gamma_0} e^{\omega|\eta| } q_t^{\Lambda,N}(\eta) \lambda (d
\eta) = \int_{\Gamma_0} \left(1+ e^{\omega}\right)^{|\eta|}
R_t^{\Lambda,N}(\eta) \lambda (d \eta),
\end{equation}
that can be obtained by (\ref{12A}). Since $R^{\Lambda, N}_0 \in
\mathcal{G}_{\beta'}$ for any $\beta'>0$, see (\ref{secure}),
we can take $\beta' =
\beta+1$ which maximizes $T(\beta', \beta)$ given in (\ref{Qq5}).
Then for each $\beta>0$, we have that
\begin{equation}
  \label{Herf10}
 R^{\Lambda, N}_t \in \mathcal{G}_{\beta}, \qquad {\rm for}  \ \ t <
 \tau(\beta):=\frac{e^{-\beta}}{e\langle a^{+} \rangle}.
\end{equation}
Hence, $q_t^{\Lambda,N} \in \mathcal{G}_\omega$ whenever
$R_t^{\Lambda,N} \in \mathcal{G}_\beta$ with $\beta$ such that
$e^{\beta} = 1+ e^{\omega}$, cf. (\ref{Qq11}). Moreover, for such
$\omega$ and $\beta$ the right-hand side of (\ref{y13}) defines a
continuous map from $\mathcal{G}_\beta$ to $\mathcal{G}_\omega$.
\begin{lemma}
 \label{q1lm}
Given $\omega_1 > 0$ and $\omega_2 > \omega_1$, let $\beta_i$ be such that $e^{\beta_i} = e^{\omega_i} +1$,
$i=1,2$.
Then $q^{\Lambda,N}_t$ is continuously
differentiable in $\mathcal{G}_{\omega_1}$ on $[0, \tau(\beta_2))$ and the following holds
\begin{equation}
 \label{qa}
 \frac{d}{dt} q^{\Lambda,N}_t = L^{\Delta, \sigma}_{\omega_1} q^{\Lambda,N}_t.
\end{equation}
\end{lemma}
\begin{proof}
By the mentioned continuity of the map in (\ref{y13}) the continuous
differentiability of $q^{\Lambda,N}_t$ follows from the
corresponding property of $R^{\Lambda,N}_t\in
\mathcal{G}_{\beta_2}$, which it has in view of (\ref{y10}). Then the following holds
\begin{equation}
 \label{qa1}
 \left(\frac{d}{dt} q^{\Lambda,N}_t \right) (\eta) = \int_{\Gamma_0} \left( L^{\dagger,\sigma}_{\beta_1}
 R^{\Lambda,N}_t\right)(\eta\cup\xi) \lambda(d\xi)
\end{equation}
Where $L^{\dagger,\sigma}_{\beta_1}$ is the trace of
$L^{\dagger,\sigma}$ in $\mathcal{G}_{\beta_1}$. We define the
action of $\widehat{L}^{\sigma} = A^\sigma + B^\sigma$ in such a way
that
\begin{equation*}
 \langle \! \langle \widehat{L}^{\sigma} G, k \rangle \!\rangle = \langle \! \langle  G, L^{\Delta,\sigma} k \rangle \!\rangle ,
\end{equation*}
where that $L^{\Delta,\sigma}$ acts as in (\ref{Qq21}) and
(\ref{Qq22}). Then $A^\sigma$ acts as in  (\ref{z1})
with $E^{+} (y,\eta)$ replaced by $\varphi_\sigma (y) E^{+}
(y,\eta)$, and $B^\sigma$ acts as in  (\ref{Aqq4}) with $a^{+}(x-y)$
multiplied by $\varphi_\sigma (x)$. Let $G:\Gamma_0\to \mathbb{R}$
be bounded and continuous. Then for some $C>0$ we have, see
(\ref{7A}),
\[
|\widehat{L}^\sigma G(\eta)| \leq |\eta|^2 C \sup_{\eta\in
\Gamma_0}|G(\eta)|,  \quad |K(\widehat{L}^\sigma G)(\eta)| \leq
|\eta|^22^{|\eta|} C \sup_{\eta\in \Gamma_0}|G(\eta)|,
\]
and hence we can calculate the integrals below
\begin{equation}
 \label{qa3}
\langle \! \langle \widehat{L}^{\sigma} G, q_t^{\Lambda,N} \rangle
\!\rangle = \langle \! \langle  G,
L^{\Delta,\sigma}_{\omega_1}q_t^{\Lambda,N} \rangle \!\rangle,
\end{equation}
where $\omega_1$ and $q_t^{\Lambda,N}$ are as in (\ref{qa1}). On the
other hand, by (\ref{qa1}) we have
\begin{gather}
\label{qa4} \langle \! \langle  \widehat{L}^{\sigma} G,
q_t^{\Lambda,N} \rangle \!\rangle =
\langle \! \langle K \widehat{L}^{\sigma} G, R_t^{\Lambda,N} \rangle \!\rangle \\[.2cm] =
\langle \! \langle K G, L^{\dagger,\sigma}_{\beta_1} R_t^{\Lambda,N}
\rangle \!\rangle  = \langle \! \langle G, \frac{d}{dt}
q^{\Lambda,N}_t \rangle \!\rangle. \nonumber
\end{gather}
Since (\ref{qa3}) and (\ref{qa4}) hold true for all bounded
continuous functions, we have that the expression on both sides of (\ref{qa})
are equal to each other, which completes the proof.
\end{proof}
\begin{corollary}
 \label{qaco}
Let $k_t^{\Lambda,N}\in \mathcal{K}_{\alpha_2}$ be the solution of
the problem (\ref{Cauchy-sigma}) with $k^{\Lambda,N}_0 =
q^{\Lambda,N}_0\in \mathcal{K}_{\alpha_1}$, see Lemma \ref{sQQlm}. Then for
each $G\in B_{\rm bs}^{\star}(\Gamma_0)$ and
\[
 t < \min\{T(\alpha_2, \alpha_1;B^\Delta); 1/ e \langle a^{+} \rangle\},
\]
see
(\ref{Herf10}), we have that
\begin{equation}
 \label{qa5}
\langle \! \langle  G, k_t^{\Lambda,N} \rangle \!\rangle \geq 0.
\end{equation}
\end{corollary}
\begin{proof}
By (\ref{y8}) and (\ref{y172}) we have that $q_0^{\Lambda,N}\in
\mathcal{U}_{\sigma, \alpha_1}$ (this is the reason to consider such
local evolutions). Let then $u_t$ be the solution as in Lemma
\ref{uQQlm} with this initial condition.
 Then by Corollaries
 \ref{uQQco} and \ref{qQQco} it follows that $k^{\Lambda,N}_t = u_t = q_{t}^{\Lambda,N}$ for the mentioned values of $t$. Thus,
the validity of (\ref{qa5}) follows by (\ref{Qq10}).
\end{proof}

\subsection{Taking the limits}

\label{SS5.4}

Note that (\ref{qa5}) holds for
\[
k_t^{\Lambda,N} = Q^\sigma_{\alpha\alpha_1} (t;B^{\Delta,\sigma}_b)
q_0^{\Lambda,N},
\]
with $\alpha\in (\alpha_1, \alpha_2)$ dependent on $t$, see (\ref{QQ2}).
In this subsection, we first pass in (\ref{qa5}) to
the limit $\sigma \downarrow 0$, then we get rid of the locality
imposed in (\ref{y8}).
\begin{lemma}
  \label{qblm}
Let $k_t$ and $k_t^\sigma$ be the solutions of the problems
(\ref{Cauchy}) and (\ref{Cauchy-sigma}), respectively, with
$k_t|_{t=0}= k^\sigma_t|_{t=0} = k_0 \in \mathcal{K}_{\alpha_0}$,
$\alpha_0 > -\log \vartheta$. Then for each $\alpha> \alpha_0$ there
exists $\widetilde{T}=\widetilde{T}(\alpha, \alpha_0) < T(\alpha,
\alpha_0;B^\Delta_b)$ such that for each $G\in B_{\rm bs}(\Gamma_0)$
and $t\in [0, \widetilde{T}]$ the following holds
\begin{equation}
  \label{qb1}
  \lim_{\sigma \downarrow 0} \langle \! \langle  G, k_t^{\sigma} \rangle
  \!\rangle = \langle \! \langle  G, k_t \rangle
  \!\rangle.
\end{equation}
\end{lemma}
\begin{proof}
Take  $\alpha _2 \in (\alpha_0, \alpha)$ and $\alpha_1 \in
(\alpha_0, \alpha_2)$. Thereafter, take
\begin{equation}
 \label{qx1}
\widetilde{T} < \min\left\{ T(\alpha_1, \alpha_0;B^\Delta_b);
T(\alpha, \alpha_2;B^\Delta_b) \right\}.
\end{equation}
For $t\leq \widetilde{T}$, by (\ref{bed7}), (\ref{fra}),
(\ref{Qq22}), and (\ref{QQ2}) we then have that the following holds,
see (\ref{fra}) and (\ref{QQ15}),
\begin{eqnarray*}
Q_{\alpha \alpha_0}(t;B^\Delta_b) k_0 & =  &
Q^\sigma_{\alpha \alpha_0} (t) k_0 + M_\sigma (t) + N_\sigma (t),
\nonumber \\[.2cm]
M_\sigma (t) & := & \int_0^t Q_{\alpha\alpha_2}
(t-s;B^\Delta_b) \left((A_2^{\Delta})_{\alpha_2 \alpha_1} - (A_2^{\Delta,\sigma})_{\alpha_2 \alpha_1}\right) k^\sigma_s d s \nonumber \\[.2cm]
N_\sigma (t) & :=  & \int_0^t Q_{\alpha\alpha_2} (t-s;B^\Delta_b)
\left((B_{2,b}^{\Delta})_{\alpha_2 \alpha_1} -
(B_{2,b}^{\Delta,\sigma})_{\alpha_2 \alpha_1}\right) k^\sigma_s d s
, \nonumber
\end{eqnarray*}
where
\begin{equation}
 \label{qx2}
k_s^\sigma =  Q^\sigma_{\alpha_1 \alpha_0} (s;B^\Delta_b) k_0.
\end{equation}
Then
\begin{gather}
  \label{Aug}
\langle \! \langle  G, k_t \rangle
  \!\rangle - \langle \! \langle  G, k_t^\sigma \rangle
  \!\rangle = \langle \! \langle  G, M_\sigma(t) \rangle
  \!\rangle + \langle \! \langle  G, N_\sigma(t) \rangle
  \!\rangle.
\end{gather}
By (\ref{Qq1}) we get
\begin{gather}
  \label{qb3}
\langle  \! \langle  G, M_{\sigma} (t)\rangle
  \!\rangle = \int_0^t \langle  \! \langle  G, Q_{\alpha \alpha_2} (t-s;B^\Delta_b) v_s \rangle
  \!\rangle d s \\[.2cm] = \int_0^t \langle  \! \langle H_{\alpha_2 \alpha}(t-s) G, v_s \rangle
  \!\rangle d s = \int_0^t \langle  \! \langle  G_{ t-s}, v_s \rangle
  \!\rangle d s ,\nonumber
\end{gather}
where
\begin{eqnarray}
 \label{qb4}
 & & \langle  \! \langle  G_{ t-s}, v_s \rangle
  \!\rangle \\[.2cm] & & \quad  = \int_{\Gamma_0} G_{ t-s} (\eta) \sum_{x\in \eta} \left(1- \varphi_\sigma (x)\right) E^{+} (x,\eta\setminus x)
k^\sigma_s (\eta \setminus x) \lambda ( d \eta)\qquad \nonumber \\[.2cm] &  & \quad = \int_{\Gamma_0} \int_{\mathbb{R}^d} G_{ t-s} (\eta\cup x)
\left(1- \varphi_\sigma (x)\right) E^{+} (x,\eta) k^\sigma_s (\eta )
d x \lambda ( d \eta),  \nonumber
\end{eqnarray}
where the latter line was obtained by means of (\ref{12A}). Note
that $k_s^\sigma \in \mathcal{K}_{\alpha_1}$ and $G_{ t-s}\in
\mathcal{G}_{\alpha_2}$ for $s\leq t\leq \widetilde{T}$, see
(\ref{qx1}). We use this fact to prove that
\begin{equation*}
g_{s} (x) := \int_{\Gamma_0} \frac{1}{|\eta| + 1} \left\vert
G_{s}(\eta \cup x) \right\vert e^{\alpha_2 |\eta|} \lambda ( d \eta)
\end{equation*}
lies in $L^1(\mathbb{R}^d)$ for each $s\in [0,t]$. Indeed, by
(\ref{12A}) and (\ref{NoR}) we get
\begin{equation}
 \label{qb6}
 \|g_{s}\|_{L^1(\mathbb{R}^d)} \leq e^{-\alpha_2} |G_{s}|_{\alpha_2} \leq C_1 <\infty,
\end{equation}
where
\begin{equation}
 \label{qb9}
 C_1 :=
e^{-\alpha_2}\max_{s\in [0,\widetilde{T}]} |G_s|_{\alpha_2} \leq
 \frac{e^{-\alpha_0}T(\alpha, \alpha_2;B^\Delta_b) |G|_{\alpha}}{T(\alpha,
 \alpha_2;B^\Delta_b)-\widetilde{T}},
\end{equation}
see (\ref{qx}). By (\ref{QQ3}) and (\ref{qx2}) we also get
\begin{equation}
\label{qb7} \max_{s\in [0,\widetilde{T}]} \|k_s^\sigma\|_{\alpha_2}
\leq \frac{T(\alpha_1, \alpha_0;B^\Delta_b)
\|k_0\|_{\alpha_0}}{T(\alpha_1, \alpha_0;B^\Delta_b)-\widetilde{T}}
=: C_2 <\infty,
\end{equation}
see (\ref{qx1}). Now we use (\ref{qb4}), (\ref{qb6}), (\ref{qb7})
and obtain by (\ref{14A}) and (\ref{Iwona}) that the following holds
\begin{eqnarray}
 \label{qb8}
& & \left\vert \langle  \! \langle  G, M_{\sigma} (t)\rangle
  \!\rangle \right\vert  \leq \varkappa(\alpha_2 - \alpha_1) \|a^{+}\|e^{\alpha_1}  C_2 \\[.2cm]
& & \qquad \times  \int_0^{\widetilde{T}}\int_{\mathbb{R}^d} g_s (x)
(1- \varphi_\sigma(x)) d x ds, \nonumber
\end{eqnarray}
where
\[
\varkappa (\beta) := \frac{1}{e\beta} + \left(\frac{2}{e\beta}
\right)^2.
\]
By (\ref{qb6}) and (\ref{qb9}) we conclude that the integrand in the
right-hand side of (\ref{qb8}) is bounded by $C_1$. By the Lebesgue
dominated convergence theorem this yields ${\rm RHS}(\ref{qb8}) \to
0$ as $\sigma\downarrow 0$. In the  same way one proves that also
\[
\left\vert \langle  \! \langle  G, N_{\sigma} (t)\rangle
  \!\rangle \right\vert \to 0, \qquad \sigma\downarrow 0,
\]
which yields (\ref{qb1}), see (\ref{Aug}).
\end{proof}
\vskip.1cm \noindent Below by a cofinal sequence
$\{\Lambda_n\}_{n\in \mathbb{N}} \subset \mathcal{B}_{\rm
b}(\mathbb{R}^d)$ we mean a sequence such that: $\Lambda_{n}\subset
\Lambda_{n+1}$ for all $n$, and each $x\in \mathbb{R}$ belongs to a
certain $\Lambda_n$.
\begin{lemma}
  \label{lastlm}
Let $\{\Lambda_n\}_{n\in \mathbb{N}}$ be a cofinal sequence and
$q_0^{\Lambda,N}$ be as in (\ref{y172}). Let also $\alpha_1$ and
$\alpha_2$ be as in Lemma \ref{Wolflm}. Then for each $t\in
[0,T(\alpha_2, \alpha_1;B^\Delta_b))$ and $G\in B_{\rm
bs}(\Gamma_0)$, the following holds
\begin{equation}
  \label{Augg}
  \lim_{n \to +\infty}\lim_{N
\to +\infty} \langle  \! \langle  G,
Q_{\alpha_2\alpha_1}(t;B^\Delta_b) q_{0}^{\Lambda_n,N}\rangle
  \!\rangle  = \langle  \! \langle  G,
Q_{\alpha_2\alpha_1}(t;B^\Delta_b )k_{\mu_0}\rangle
  \!\rangle.
\end{equation}
\end{lemma}
\begin{proof}
As in (\ref{qb3}), we prove (\ref{Augg}) by showing that
\begin{gather}
\label{Augs} \lim_{n \to +\infty}\lim_{N \to +\infty} \langle  \!
\langle  G, Q_{\alpha_2\alpha_1}(t;B^\Delta_b)
q_{0}^{\Lambda_n,N}\rangle
  \!\rangle \\[.2cm] = \lim_{n \to +\infty}\lim_{N
\to +\infty} \langle  \! \langle H_{\alpha_1 \alpha_2} (t)G,
 q_{0}^{\Lambda_n,N}\rangle
  \!\rangle =  \langle  \! \langle H_{\alpha_1 \alpha_2}(t) G,
 k_{\mu_0}\rangle\!\rangle. \nonumber
\end{gather}
Since $G_t := H_{\alpha_1 \alpha_2}(t) G$ lies in
$\mathcal{G}_{\alpha_1}$, the proof of (\ref{Augs}) can be done by
the repetition of arguments used in the proof of the analogous
result in \cite[Appendix]{Berns}.
\end{proof}

\subsection{The proof of Lemma \ref{Wolflm}}

Let $\alpha_1$ and $\alpha_2$ be as in Lemma \ref{Wolflm} and
$\{\Lambda_n\}_{n\in \mathbb{N}}$ be a cofinal sequence. Take
$k_{\mu_0} \in \mathcal{K}_{\alpha_1}$ and then produce
$q_0^{\Lambda_n,N}$, $n\in \mathbb{N}$, by employing (\ref{y8}) and
(\ref{y172}). Let $T(\alpha_2, \alpha_1)< T(\alpha_2,
\alpha_1;B^\Delta_b)$ be such that (\ref{qa5}) holds with
\[
k_t^{\Lambda_n,N} = Q^\sigma(\alpha_2,\alpha_1;B^\Delta_b)
q_0^{\Lambda_n,N}, \qquad t\leq T(\alpha_2, \alpha_1).
\]
Note that $T(\alpha_2, \alpha_1)$ is independent of $\Lambda_n$ and
$N$, see Corollary \ref{qaco}. By Lemma \ref{lastlm} we then have
that
\[
\langle  \! \langle  G, Q^\sigma_{\alpha_2\alpha_1}(t;B^\Delta_b
)k_{\mu_0}\rangle
  \!\rangle \geq 0.
\]
Now we apply Lemma \ref{qblm} and obtain
\[
\langle  \! \langle  G, Q_{\alpha_2\alpha_1}(t;B^\Delta_b
)k_{\mu_0}\rangle
  \!\rangle \geq 0,
\]
which holds for $$t\leq \tau(\alpha_2, \alpha_1) :=\min\left\{
T(\alpha_2, \alpha_1); \widetilde{T}(\alpha_2, \alpha_1) \right\},
$$
which completes the proof.

\section*{Acknowledgment}
This paper was partially supported by the DFG through the SFB 701
``Spektrale Strukturen und Topologische Methoden in der Mathematik''
and by the European Commission under the project STREVCOMS
PIRSES-2013-612669.

\end{document}